\documentclass[12pt]{amsart}

\usepackage[T1]{fontenc}
\usepackage{fullpage}
\usepackage{amsmath,amssymb,mathtools,mathrsfs}
\usepackage{algorithm,algpseudocode}
\usepackage{longtable,array}
\usepackage[dvipsnames]{xcolor}
\usepackage{microtype}
\usepackage[
  colorlinks=true,
  linkcolor=BrickRed,
  citecolor=Blue,
  urlcolor=Blue
]{hyperref}
\numberwithin{equation}{section}
\theoremstyle{plain}
\newtheorem{thm}{Theorem}[section]
\newtheorem{lem}[thm]{Lemma}
\newtheorem{prop}[thm]{Proposition}

\theoremstyle{definition}

\theoremstyle{remark}

\newcommand{\CC}{\mathbb C}
\newcommand{\ZZ}{\mathbb{Z}}

\newcommand{\PP}{\mathbb P}
\newcommand{\DD}{\mathbb D}
\newcommand{\calF}{\mathcal F}
\newcommand{\calH}{\mathcal H}
\newcommand{\calM}{\mathcal M}
\DeclareMathOperator{\Aut}{Aut}
\DeclareMathOperator{\GL}{GL}
\DeclareMathOperator{\PGL}{PGL}
\DeclareMathOperator{\SL}{SL}

\DeclareMathOperator{\Sym}{Sym}
\DeclareMathOperator{\diag}{diag}

\DeclareMathOperator{\tr}{tr}

\newcommand{\Auts}{\Aut^{\mathrm s}}
\DeclareMathOperator{\ord}{ord}
\newcommand{\disc}{\operatorname{disc}}
\newcommand{\rk}{\operatorname{rk}}

\title{Remarks on Lattice Actions Corresponding to Automorphism Groups of Smooth Cubic Fourfolds}
\author[J. Fu]{Jie Fu}
\address{Qiuzhen College, Tsinghua University, China}
\email{fu-j21@mails.tsinghua.edu.cn}
\date{}

\begin{document}
\raggedbottom
\allowdisplaybreaks[4]
\bibliographystyle{amsalpha}

\begin{abstract}
The automorphism groups of smooth cubic fourfolds over the complex numbers, together with actions on defining equations, are classified in previous works. In this paper, we complete the calculation of group actions on the primitive integral cohomology lattice for all of the $156$ families in the classification. We distinguish similar families, determine the remaining action-maximal cases, and use restriction to complete the list. We also determine the generic algebraic lattices for each family. The calculation depends mainly on OSCAR.
\end{abstract}

\maketitle
\setcounter{tocdepth}{1}
\tableofcontents

\medskip

\section{Introduction}\label{sec:introduction}
The classification of automorphism groups of smooth cubic fourfolds begins with the classification of symplectic groups and their coinvariant lattices by Laza--Zheng \cite{laza2022automorphisms}, using the results of H\"ohn--Mason \cite{hohn2016290}. Koike gives the corresponding projective actions and invariant cubics \cite{KOIKE202512,KOIKE2026}, and Yang--Yu--Zhu classify the finite groups acting faithfully on smooth cubic fourfolds \cite{yang2024automorphism}. The non-symplectic indices are studied in \cite{FWZ26}, and the classification of possible lattice actions is developed in \cite{FZ26}. The resulting classification in \cite{FWZclass} consists of $156$ connected families. We write $\mathscr{F}_i$ for the $i$-th family in \cite[Table~1]{FWZclass}.

We complete the lattice calculations of \cite{FZ26} and match the resulting actions with the geometric families in \cite{FWZclass}.

For a smooth cubic fourfold $X$ over the complex numbers, write $\Lambda$ and $\Lambda_0$ for the middle integral cohomology and its primitive part, respectively. By the global Torelli theorem, $\Aut(X)$ is isomorphic to the group of integral Hodge isometries of $\Lambda$ fixing the square of the hyperplane class. In this note, we finish the calculation of group actions on these lattices and of the sublattices $S,T,P,K$ and $A$:
\begin{thm}\label{thm:catalogue}
For each of the $156$ families in \cite[Table~1]{FWZclass}, we determine the action of the automorphism group of a very general member on $\Lambda_0$, up to conjugacy in $O(\Lambda_0)$, together with the primitive sublattices $S,T,P,K\subset\Lambda_0$. These data also determine the generic algebraic lattice $A\subset\Lambda$.
\end{thm}

See \S~\ref{sec:preliminaries} for explicit definitions of the lattices $S,T,P,K$ and $A$.

Many cases with distinguishable groups have already been solved in \cite{FZ26}. We further distinguish the two cubic fourfolds both with symplectic group $M_{10}$, the two families both with automorphism group $L_2(7):2$, and similarly for $S_{3,3}$ and $A_{3,3}$. For five further action-maximal families, we enumerate the lattice actions using characteristic polynomials calculated from the defining equations. We construct the action for the action-maximal family with automorphism group $C_2\times C_4$, and obtain further actions by restriction.

The paper is organized as follows. \S\ref{sec:preliminaries} recalls the conventions of lattices, period maps, and previous results, and gives the formulas for calculating characteristic polynomials. \S\ref{sec:two-pairs} distinguishes similar families that can not be identified using plain information like abstract structure of (symplectic) automorphism group or dimension. \S\ref{sec:maximal} determines the six remaining action-maximal cases by enumeration or construction. \S\ref{sec:restriction} obtains the other actions by restriction and treats the special $A_4\times C_2$ family separately. \S\ref{sec:data} lists the discriminant forms of lattices associated to each family.

The lattice computations use OSCAR \cite{OSCAR}. The computational files are archived in \cite{FuWangZhengAutcub4fold2026}. The integral actions and algebraic lattices are in \path{remark/catalogue/assembly/lattice_156_with_algebraic.mrdi}. The Gram matrices and the forms $q_P,q_K$ are listed in \path{remark/catalogue/assembly/lattice_156_gram_matrices.txt}. The ancillary files include copies of these data and \path{anc/lattice_156_actions.txt}, which gives bases of the embedded lattices and matrices of their compatible actions, including generators of the symplectic groups.

\medskip\noindent\textbf{Acknowledgements.} The author thanks Zhiwei Zheng for organizing a seminar a year ago and encouraging him to study lattice theory and take part in this series of studies. The author also thanks Shihao Wang for their previous collaboration and for his interest in this paper and helpful discussions, Stevell Muller for his help and advice on OSCAR.jl, and the other contributors to OSCAR.jl. The author is grateful to ICERM and Brown University's Center for Computation and Visualization for access to the Oscar computing cluster, where the similarly named OSCAR.jl was used for some calculations. Part of this work was carried out at ICERM, Brown University, during the Fall 2026 program ``Computations on K3 Surfaces and Related Varieties'', with support through National Science Foundation grant DMS-2424556.

\medskip\noindent\textbf{AI disclosure.} Large language models were used to explore proof strategies, assist in the development and refinement of mathematical arguments, assist in writing and organizing code, and draft and revise the text. The author has checked all mathematical arguments and the full text, and takes responsibility for all results and arguments presented in this paper, as well as for the code and accompanying documentation.

\section{Preliminaries}\label{sec:preliminaries}

\subsection{Lattices and periods}

For an integral lattice $L$, write $L^\vee$ for its dual lattice, $A_L=L^\vee/L$ for its discriminant group, and $\disc(L)$ for its discriminant. Then $|A_L|=|\disc(L)|$. For even $L$, write $q_L$ for its discriminant quadratic form.

Write $O(L)$ for the isometry group of $L$ and, when $L$ is even, $\widetilde O(L)$ for the kernel of $O(L)\to O(q_L)$. For a primitive vector $v\in L$, let $\operatorname{div}_L(v)=\gcd\{(v,x):x\in L\}$.

Let $X=V(F)\subset\PP^5$ be a smooth cubic fourfold, and let $h$ be its hyperplane class. Write $\Lambda=H^4(X,\ZZ)$ and $\Lambda_0=(h^2)^\perp_\Lambda$. The lattice $\Lambda$ is odd and unimodular, and $\langle h^2\rangle\cong\langle3\rangle$. We take $E_8$ and $A_2$ to be positive definite. Then
\[
\Lambda_0=(h^2)^\perp_{\Lambda}
\cong U^{\oplus2}\oplus E_8^{\oplus2}\oplus A_2
\]
has signature $(20,2)$ and discriminant $3$ \cite[\S2.1]{laza2022automorphisms}.

Every automorphism of $X$ is induced by a projective linear transformation \cite{Matsumura1963OnTA}.

Recall the global Torelli theorem for cubic fourfolds, originally shown in \cite{voisin1986torelli,voisin2008erratum}, with an alternative proof provided by \cite{looi2009period}. Let $X'$ be another smooth cubic fourfold with hyperplane class $h'$. Any Hodge isometry $H^4(X',\ZZ)\xrightarrow{\sim}H^4(X,\ZZ)$ sending $(h')^2$ to $h^2$ is induced by a unique isomorphism $X\xrightarrow{\sim}X'$.

Taking $X'=X$, we obtain
\[
\Aut(X)\cong\Aut_{\mathrm{HS}}(\Lambda,h^2),
\]
where $\Aut_{\mathrm{HS}}(\Lambda,h^2)$ is the group of integral Hodge isometries fixing $h^2$; see also \cite[Proposition~1.3]{zheng2021orbifold}.

Following \cite[\S2.1]{laza2022automorphisms}, let $\DD$ be one component of the domain of isotropic lines $[\sigma]\in\PP(\Lambda_0\otimes\CC)$ with $(\sigma,\overline\sigma)<0$. Write $O^+(\Lambda_0)$ for the subgroup preserving $\DD$. A vector in $\Lambda_0$ is a short root if it has square $2$, and a long root if it has square $6$ and divisibility $3$. Let $\calH_6$ and $\calH_2$ be the hyperplane arrangements defined by short and long roots, respectively.

Let $\calM$ be the moduli space of smooth cubic fourfolds. The period map gives an isomorphism of quasi-projective varieties \cite{laza2010period,looi2009period}
\[
\calM\cong\bigl(\DD\setminus(\calH_2\cup\calH_6)\bigr)/
\bigl(\widetilde O(\Lambda_0)\cap O^+(\Lambda_0)\bigr).
\]

Let $G=\Aut(X)$, and let $G^s=\Auts(X)$ be the subgroup acting trivially on $H^{3,1}(X)$. Write $T_X=\Lambda_0^{G^s}$ and $S_X=(T_X)^\perp_{\Lambda_0}$, and omit $X$ when it is fixed. The lattice $S$ is positive definite, while $T$ has signature $(20-\rk S,2)$.

By Nikulin's results for primitive embeddings \cite[Proposition~1.5.1 and Corollary~1.5.2]{Nikulin_1980}, the quotient $\Lambda_0/(S\oplus T)\subset A_S\oplus A_T$ is the graph of an anti-isometry between subgroups of $A_S$ and $A_T$. Every isometry in $\widetilde O(\Lambda_0)$ extends to $\Lambda$ fixing $h^2$.

The quotient $G/G^s$ is cyclic, and its order is the non-symplectic index \cite{FWZ26}. Choose $f\in G$ mapping to a generator of $G/G^s$, and also write $f\in\widetilde O(\Lambda_0)$ for its induced isometry. We call this an \emph{extra isometry}, and write $f_T=f|_T$. Then $\ord(f_T)=[G:G^s]$. If $G=G^s$, we choose $f=\mathrm{id}$.

A \emph{fitting extension} of $f_T\in O(T)$ is an isometry of $\Lambda_0$ preserving $S,T$ and restricting to $f_T$ on $T$ \cite[\S5.1]{FZ26}.

We regard coordinates $x$ as a column vector and write $\Aut(F)=\{M\in\GL(6,\CC):F(Mx)=F(x)\}$. A \emph{strict linear group} is a finite subgroup $H\subset\Aut(F)$ containing $\mu_3 I_6$.

For the classified family $\mathscr{F}_i$, let $H_i$ be its associated linear group and write $G_i=H_i/(\mu_3 I_6)\subset\PGL(6,\CC)$. The symplectic subgroup of $G_i$ is $G_i^s=H_i^s/(\mu_3 I_6)$, where $H_i^s=H_i\cap\SL(6,\CC)$. Write $W_i=\Sym^3((\CC^6)^*)^{H_i}$, and let $W_i^{\mathrm{sm}}$ be the nonempty locus of forms defining smooth cubic fourfolds. Then $G_i\subset\Aut(V(F))$ for every $F\in W_i^{\mathrm{sm}}$. The family is \emph{saturated} if equality holds for a general member \cite[Definition~3.4(1)]{FZ26}.

Following \cite[\S2A--B]{yu2020moduli} and \cite[\S3.1]{FZ26}, the moduli space with this prescribed symmetry is the geometric quotient
\[
\calF_i=W_i^{\mathrm{sm}}/N_{\GL(6,\CC)}(H_i).
\]
Let $\pi_i:\calF_i\to\calM$ forget the group action, and write $\mathcal Z_i=\pi_i(\calF_i)$ with its reduced induced structure. The map $\calF_i\to\mathcal Z_i$ is finite and surjective, and is the normalization morphism when the family is saturated \cite[Proposition~2.7]{yu2020moduli}. Thus $\dim\calF_i=\dim\mathcal Z_i$. We keep $\mathscr{F}_i$ for the classified family and express incidence relations using its moduli image $\mathcal Z_i$.

By the Hodge--Riemann relations, $(H^{3,1}(X)\oplus H^{1,3}(X))\cap\Lambda_{0,\mathbb R}$ is negative definite of dimension two, and the real primitive $(2,2)$-part is positive definite. The marked symmetric period domain has two components of type IV for non-symplectic index one or two, and is a complex ball otherwise \cite[Propositions~4.1 and 4.9]{yu2020moduli}.

Write $A_X=\Lambda\cap H^{2,2}(X)$ for the algebraic lattice \cite{hassett2000special}, and let $K_X=A_X\cap\Lambda_0$ and $P_X=(K_X)^\perp_{\Lambda_0}$. Then
\[
A_X=(P_X)^\perp_\Lambda=\operatorname{Sat}_{\Lambda}(\langle h^2\rangle\oplus K_X).
\]
Let $j=|A_X/(\langle h^2\rangle\oplus K_X)|$. Unimodularity of $\Lambda$ and $(h^2)^2=3$ give $j\in\{1,3\}$ and $j^2\disc P_X=3\disc K_X$. When $j=1$, $A_X\cong\langle3\rangle\oplus K_X$.

\subsection{Previous results}

There are $48$ connected symplectic families \cite[Proposition~4.2]{FZ26}. The symplectic subgroup acts as $\widetilde O(S)$ on $S$ and trivially on $T$ \cite{laza2022automorphisms}. The image of $G^s$ and the extra isometry $f$ generate the image of $G$ in $O(\Lambda_0)$. Recovering this embedded action requires the primitive embedding and a fitting extension of the action on $T$.

On the connected smooth locus of cubics fixed by $H$, its action on $\Lambda_0$ is locally constant up to conjugacy. In a local marking, additional rational classes have type $(2,2)$ only on a countable union of proper analytic subsets. Thus the lattices $S_X,T_X,A_X,K_X,P_X$ of a saturated family are independent of a very general member, up to integral conjugacy. We omit $X$ for these generic lattices. Lattice-theoretic equivalence also preserves the action on $H^{3,1}$ \cite[Definition~3.4(3)]{FZ26}.

The enumeration in \cite[Theorem~1.1 and Proposition~5.8]{FZ26} gives the groups, non-symplectic indices, and compatible lattice actions for $15\leq\rk S\leq19$. The same enumeration of lattice actions applies in lower rank, but the larger rank of $T$ makes it expensive.

Write $\zeta_m=e^{2\pi i/m}$. Let $f\in\widetilde O(\Lambda_0)$ preserve $S,T$, with $f_T$ of order $m$, and let $P=\ker\Phi_m(f_T)$ and $K=P^\perp_{\Lambda_0}$. If the real subspace of $T\otimes\mathbb R$ corresponding to $\zeta_m,\zeta_m^{-1}$ has signature $(*,2)$ and $K$ contains neither short nor long roots, the $\zeta_m$-eigenspace of $f_T$ contains smooth periods, and $f$ is induced by an automorphism \cite[Proposition~4.4]{FZ26}. For a very general period, $P$ and $K$ are the transcendental and primitive algebraic lattices \cite[Proposition~4.7]{FZ26}. We check the generic automorphism group separately.

All $156$ families are self-conjugate \cite[Proposition~6.6]{FWZclass}. Thus the two conjugate components of type IV for a fixed lattice action give equivalent families \cite[Proposition~3.6]{FZ26}.

\subsection{Characteristic polynomials}

Let $M\in\GL(6,\CC)$ satisfy $F(Mx)=\eta(M)F(x)$. Griffiths' residue formula \cite{griffiths1969rational} gives its eigenvalue on $H^{3,1}(X)$ as $\chi(M)=\det(M)\eta(M)^{-2}$. On $\Aut(F)$, this is $\chi(M)=\det(M)$, and cubic scalar matrices act trivially.

Let $R_F=\CC[x_1,\ldots,x_6]/(\partial F)$ be the Jacobian ring. For $M\in\Aut(F)$ with eigenvalues $\lambda_1,\ldots,\lambda_6$, the partial derivatives form a regular sequence with eigenvalues $\lambda_1^{-1},\ldots,\lambda_6^{-1}$. Their Koszul resolution gives
\begin{equation}\label{eq:jacobian-character}
\sum_{d\geq0}\tr(M\mid R_{F,d})t^d
=\prod_{j=1}^6\frac{1-\lambda_j^{-1}t^2}{1-\lambda_jt}.
\end{equation}
Griffiths' residue formula gives
\[
\tr(M\mid\Lambda_0\otimes\CC)
=\det(M)\sum_{d\in\{0,3,6\}}\tr(M\mid R_{F,d}).
\]
\begin{prop}\label{prop:equation-character}
Let $H\subset\Aut(F)$ be a finite subgroup containing $\mu_3 I_6$, let $H^s=H\cap\SL(6,\CC)$, and write $T_H=\Lambda_0^{H^s}$. For $g\in H$ and $k\geq0$,
\[
\tr(g^k\mid T_H\otimes\CC)
=\frac{1}{|H^s|}\sum_{h\in H^s}
\tr(g^kh\mid\Lambda_0\otimes\CC),
\]
where each trace on the right is computed from \eqref{eq:jacobian-character} and the residue formula. These traces determine the characteristic polynomial of $g|_{T_H}$, and $\chi(g)=\det(g)$. If $H^s$ maps onto $\Auts(X)$, then $T_H=T_X$.
\end{prop}

\begin{proof}
Averaging over $H^s$ is the projector onto $T_H\otimes\CC$ and commutes with $g$. Taking its product with $g^k$ gives the stated trace. The traces of all powers determine the eigenvalue multiplicities, and hence the characteristic polynomial.
\end{proof}

\section{Distinguishing families}\label{sec:two-pairs}

There are two families with symplectic group $M_{10}$ and two with automorphism group $L_2(7):2$. We distinguish them by their primitive embeddings and the involutions induced on $T$ by non-symplectic automorphisms, respectively. The two $S_{3,3}$ and two $A_{3,3}$ symplectic families are distinguished by their groups and traces. We write $x_{ijk}$ for $x_ix_jx_k$.

\subsection{The two $M_{10}$ families}

There are exactly two smooth cubic fourfolds with symplectic group $M_{10}$ and non-symplectic index one, representing $\mathscr{F}_5$ and $\mathscr{F}_6$ \cite[Theorem~1.8(4)]{laza2022automorphisms}. They have isometric $S$ and $T$, with $T\cong\langle-12\rangle\oplus\langle-30\rangle$, but different primitive embeddings \cite[Proposition~4.2]{FZ26}. We use Koike's equations \cite[(2.3) and (8.6)]{KOIKE202512}\footnote{The coefficients $4,5,8$ in \cite[(2.8)]{KOIKE202512} should be $-4,-5,-8$, respectively.}:
\[
F^M_\pm=f-(7\pm3\sqrt6)g
\]
where
\begin{align*}
f={}&2\sum_{i=1}^6x_i^3+3(x_{112}+x_{122}+x_{334}+x_{344}+x_{556}+x_{566})\\
  &{}+2(x_{235}+x_{145}+x_{136}+x_{246})+4(x_{135}+x_{245}+x_{236}+x_{146}),\\
g={}&3(x_{123}+x_{124}+x_{125}+x_{126}+x_{134}+x_{135}+x_{156}\\
  &\qquad+x_{234}+x_{256}+x_{356}+x_{346}+x_{345}+x_{456})\\
  &{}-4(x_{145}+x_{136}+x_{235}+x_{246})-5(x_{146}+x_{236}+x_{245})-8x_{135}.
\end{align*}
Write $X^M_\pm=V(F^M_\pm)$. Let $e_1,e_2$ be the orthogonal basis of $T$ with squares $-12,-30$, respectively. For an embedding $i:S\oplus T\hookrightarrow\Lambda_0$ with primitive images of $S,T$, let $H_T\subset A_T$ be the projection of the gluing subgroup to $A_T$.

\begin{lem}\label{prop:m10}
\begin{enumerate}
    \item There exist embeddings $i_1,i_2:S\oplus T\hookrightarrow\Lambda_0$ with images of $S,T$ primitive, whose gluing subgroups satisfy $(H_T^{(1)})^\perp=\langle e_1/3\rangle$ and $(H_T^{(2)})^\perp=\langle e_2/3\rangle$. Moreover, $\operatorname{div}_{\Lambda_0}(i_1(e_1))=3$, $\operatorname{div}_{\Lambda_0}(i_2(e_1))=1$, $\disc((i_1(e_1))^\perp)=-4$ and $\disc((i_2(e_1))^\perp)=-36$.
    \item Let $X$ be one of $X^M_\pm$ and let $\iota([x_1:\cdots:x_6])=[\overline{x}_1:\cdots:\overline{x}_6]$ be complex conjugation in the coordinates of $F^M_\pm$. The pair $(S_X,T_X)$ is conjugate in $O(\Lambda_0)$ to $(i_j(S),i_j(T))$ for $j=1$ or $2$. The induced action of $\iota$ on $\Lambda$ fixes $h^2$ and preserves $T_X$. Under this identification, $T^\iota=\langle e_j\rangle$.
\end{enumerate}
\end{lem}

\begin{proof}
For (1), we have $\disc S=120$ and $\disc T=360$, so the gluing subgroup has order $120$ and projects onto $A_S$. By \cite[Proposition~1.5.1]{Nikulin_1980}, its other projection $H_T\subset A_T$ must admit an anti-isometry with $A_S$. The gluing calculation in \path{remark/catalogue/verify_m10_gluings.jl} gives exactly two possibilities for $H_T$, with $(H_T^{(1)})^\perp=\langle e_1/3\rangle$ and $(H_T^{(2)})^\perp=\langle e_2/3\rangle$. Both gluings occur in $\Lambda_0$ by the classification of primitive embeddings in \cite[\S4.5]{laza2022automorphisms}. In both cases $A_T=H_T\oplus H_T^\perp$, with $H_T^\perp\cong A_{\Lambda_0}$. The calculation also gives divisibilities $3,1$ for the image of $e_1$. Since $e_1^2=-12$, we have
\[
\disc((i_j(e_1))^\perp)
=\frac{-36}{\operatorname{div}_{\Lambda_0}(i_j(e_1))^2}.
\]
The conclusion follows.

For (2), the map $O(S)\to O(q_S)$ is surjective by \cite[Lemma~4.25]{laza2022automorphisms}, so for a fixed $H_T$ all anti-isometries $A_S\to H_T$ give equivalent gluings. Thus the embedded pair $(S_X,T_X)$ is conjugate to one of the two pairs in (1). We use this identification for the rest of the proof.

The map $\iota$ is antiholomorphic and preserves the real orientation because $\dim_{\CC}X=4$, so its pullback preserves the cup product on $\Lambda$. Since $\iota$ normalizes $\Auts(X)$, it preserves $T_X$. Its induced complex-linear action exchanges $H^{3,1}$ and $H^{1,3}$, so it reverses the orientation of $T_X\otimes\mathbb R$. The only such isometries of $\langle-12\rangle\oplus\langle-30\rangle$ reversing orientation are $\pm\operatorname{diag}(1,-1)$. Also, $\iota(h)=-h$ and $\iota(h^2)=h^2$. Since $\Lambda$ is unimodular, $\iota$ acts trivially on $q_{\Lambda_0}$, and further on $H_{T_X}^\perp=\langle e_j/3\rangle$. Therefore $\iota$ fixes $e_j$ and negates the other basis vector of $T$, which implies the conclusion.
\end{proof}

\begin{prop}\label{prop:m10-identification}
\begin{enumerate}
    \item The cubic $X^M_+$ is linearly equivalent to $X^1(M_{10})$ in \cite[(4.3)]{laza2022automorphisms}, and $X^M_-$ to its Galois conjugate under $\zeta_{24}\mapsto\zeta_{24}^7$.
    \item Let $X_t=V(f-tg)$ for $t\in \CC$. Then $X_t$ is smooth for $7-3\sqrt{6}\leq t\leq 0$.
    \item Write $\iota_0,\iota_+,\iota_-$ for coordinate-wise complex conjugation on $X_0,X^M_+,X^M_-$, respectively. Under identifications $T_{X^M_\pm}\cong T$, we have $(T_{X^M_+})^{\iota_+}=\langle e_1\rangle$ and $(T_{X^M_-})^{\iota_-}=\langle e_2\rangle$. The pairs $(S_X,T_X)$ for $X=X^M_+$ and $X=X^M_-$ are conjugate in $O(\Lambda_0)$ to $(i_1(S),i_1(T))$ and $(i_2(S),i_2(T))$, respectively.
\end{enumerate}
\end{prop}

\begin{proof}
(1) is checked in \path{remark/catalogue/verify_m10_lz_koike_conjugacy.g}. It also checks that $f,g$ are invariant under the two matrices in \cite[(2.7)]{KOIKE202512}.

For (2), we show that if $X_t$ is singular for $t\in\CC$, then
\[
(t+1)(4t-1)(t-7)(8t^2+5t+5)=0.
\]
This equation has no root in $[7-3\sqrt6,0]$.

A positive-dimensional singular locus must meet the hyperplane $x_1+2x_2+\cdots+6x_6=0$. In each chart $x_k=1$, eliminating the coordinates from the six partial-derivative equations and this hyperplane equation gives $\langle8t^2+5t+5\rangle$; see \path{remark/catalogue/m10_surface_slope_exact.sing}. Thus $8t^2+5t+5=0$ in this case.

Otherwise the singular locus is finite. Five general linear combinations of the quadratic partial derivatives have a finite common zero set, so B\'ezout's theorem bounds the number of singular points by $2^5=32$. Thus a singular point has a stabilizer of index at most $32$ in $A_6$, whose preimage in $3.A_6$ preserves the line in $\CC^6$ representing that point. For each conjugacy class of such subgroups, \path{remark/catalogue/m10_projective_subgroups.g} computes the common eigenlines and solves $\nabla f=t\nabla g$ on them. The only finite parameters obtained are $t=-1,1/4,7$.

For (3), the cubic $X_0=V(f)$ is the non-Clebsch $A_7$ cubic, with invariant lattice of rank two
\[
T_{X_0}\cong-\begin{bmatrix}18&3\\3&18\end{bmatrix}
\]
by \cite[Theorem~1.8(2)]{laza2022automorphisms}. Let $u_1,u_2$ be a basis with this Gram matrix, and write $v=u_1-u_2$, $w=u_1+u_2$. Then $v^2=-30$ and $w^2=-42$. We have $\disc S_{X_0}=105$ and $|A_{T_{X_0}}|=315$ \cite[Table~1]{FWZ26}. 
By \cite[Proposition~3.4(2)]{FWZ26}, we have
\[
q_{T_{X_0}}\cong(-q_{S_{X_0}})\oplus3^{-1},
\]
where $3^{-1}=q_{\Lambda_0}$. Thus $A_{\Lambda_0}$ corresponds to the residual subgroup $H_{T_{X_0}}^\perp\subset A_{T_{X_0}}$, whose generator has quadratic value $2/3$. Of the four subgroups of order three in $A_{T_{X_0}}$, only $\langle[v/3]\rangle$ has this quadratic form. Hence $v/3$ represents a generator of $A_{\Lambda_0}$.

The same argument as in the proof of Lemma~\ref{prop:m10}(2) shows that $\iota_0$ fixes $[v/3]\in A_{\Lambda_0}$ and reverses the orientation of $T_{X_0}\otimes\mathbb R$. The only orientation-reversing isometries of $T_{X_0}$ are $\pm\left[\begin{smallmatrix}
    0 & 1\\
    1 & 0
\end{smallmatrix}\right]$. Combining this with $\iota_0([v/3])=[v/3]$, we obtain $\iota_0(v)=v$ and $\iota_0(w)=-w$.

The space of invariant cubic forms for the $3.A_6$ representation in \cite[(2.7) and (2.9)]{KOIKE202512} is $\CC f\oplus\CC g$. Its smooth members belong to the $A_6$ family $\mathscr{F}_{10}$, so every smooth $X_t$ has this common $A_6$ action. Let $L=\Lambda_0^{A_6}$ be its invariant lattice, of rank three and discriminant $540$ \cite[Table~1]{FWZ26}. Complex conjugation normalizes this action, so $\iota_0$ preserves $L$.

Let $U=S_{X_0}\cap L$ and $m=[L:T_{X_0}\oplus U]$. Since $\disc T_{X_0}=315$, a generator of $U$ has square $12m^2/7$, so $7\mid m$. The nontrivial $7$-primary gluing pairs the class of $w/7$ in $A_{T_{X_0}}$ with a class of order $7$ in $A_U$. The lattice $U$ is preserved by $\iota_0$ and has rank one. Since $\iota_0$ negates $w/7$, the gluing graph forces $\iota_0$ to act as $-1$ on $U$. Thus $L_{\mathbb Q}^{\iota_0}=\mathbb Qv$. The lattice $T_{X_0}$ is primitive in $L$, and $v$ is primitive in $T_{X_0}$, so $L^{\iota_0}=\langle v\rangle\cong\langle-30\rangle$.

Parallel transport along the smooth real path from $X_0$ to $X^M_-$ gives an isometry $\tau$ between the two copies of $L$, with
\[
\tau\circ\iota_0=\iota_-\circ\tau.
\]
Hence $L^{\iota_-}\cong\langle-30\rangle$. The positive orthogonal complement of $T_{X^M_-}$ in $L$ has no fixed vector, so $(T_{X^M_-})^{\iota_-}=L^{\iota_-}$. By Lemma~\ref{prop:m10}(2), $X^M_-$ corresponds to $i_2$, and $X^M_+$ to $i_1$ since the two cubics are not isomorphic \cite[\S8(vi)]{KOIKE202512}.
\end{proof}

The gluing calculation and the correspondence with the equations are recorded in \path{remark/catalogue/m10_gluings.md}.

\subsection{The two $L_2(7):2$ families}
The moduli image $\mathcal Z_{12}$ of the symplectic family $\mathscr{F}_{12}$ with group $L_2(7)$ contains exactly two isomorphism classes of cubic fourfolds with automorphism group $L_2(7):2$, corresponding to $\mathscr{F}_{13}$ and $\mathscr{F}_{14}$ \cite[Theorem~1.2(5)]{he2025cubicfourfoldsorder7automorphism}. In the model with seven coordinates of \cite[\S3.4]{he2025cubicfourfoldsorder7automorphism}, let $F_0=\sum_{i=1}^{7}z_i^3$ and
\[
F_1=z_1z_2z_4+z_2z_3z_5+z_3z_4z_6+z_1z_5z_6+z_1z_3z_7+z_4z_5z_7+z_2z_6z_7.
\]
Write $X^L_t=\{F_0+tF_1=0=\sum_i z_i\}$ and $X^L_\pm=X^L_{\pm3\sqrt2/2}$. For the common $L_2(7)$ action on the smooth members, let $T=\Lambda_0^{L_2(7)}$ and $S=T^\perp$. In a basis of $T$, its Gram matrix is
\[
\begin{bmatrix}-2&1&0\\1&10&0\\0&0&-28\end{bmatrix}.
\]
We write vectors as rows in the displayed basis.

\begin{prop}\label{prop:l27}
Let $v_1,v_2\in T$ have coordinates $(1,2,0)$ and $(7,-7,3)$, respectively, and let $r_j(x)=x-(x,v_j)v_j/21$ for $x\in T$.
\begin{enumerate}
    \item The involutions $-r_1,-r_2$ preserve $T$, with $T^{-r_j}=\langle v_j\rangle\cong\langle42\rangle$. Their anti-invariant lattices are $\langle-2\rangle\oplus\langle-28\rangle$ and $\langle-4\rangle\oplus\langle-14\rangle$, respectively. In particular, they are not conjugate in $O(T)$.
    \item Let $g_\pm\in\Aut(X^L_\pm)\setminus\Auts(X^L_\pm)$. Their restrictions to $T$ are conjugate in $O(T)$ to $-r_1$ and $-r_2$, respectively. Thus the anti-invariant lattice contains a vector of square $-2$ for $g_+|_T$, but not for $g_-|_T$.
\end{enumerate}
\end{prop}

\begin{proof}
For (1), both $v_j$ are primitive, have norm $42$, and have divisibility $21$ in $T$. Thus $-r_j$ is integral, fixes $v_j$, and acts as $-1$ on $v_j^\perp$. The saturated orthogonal complements have bases $(1,0,0),(0,0,1)$ and $(-3,1,0),(2,-2,1)$, respectively. Their Gram matrices are $\diag(-2,-28)$ and $\diag(-14,-4)$. Only the first represents $-2$, so the two involutions are not conjugate. The eigenlattices of the two saved actions are checked in \path{remark/catalogue/verify_l27_split.jl}.

For (2), the restrictions of $g_\pm$ to $T$ do not depend on their choice, since $\Auts(X^L_\pm)$ acts trivially on $T$. Write $\iota^L_0,\iota^L_+$ for coordinate-wise complex conjugation on $X^L_0,X^L_+$, respectively. The cubic $X^L_0$ is the Clebsch--Segre fourfold, with
\[
T_{X^L_0}\cong-\begin{bmatrix}2&1\\1&18\end{bmatrix}
\]
by \cite[Corollary~6.9]{laza2022automorphisms}. Its only vectors of square $-2$ are $e$ and $-e$. Let $r_e\in O(T)$ be reflection in $e$, fixing $U=(T_{X^L_0})^\perp\cap T$. The restriction of $\iota^L_0$ to $T_{X^L_0}$ reverses orientation, so it is $r_e|_{T_{X^L_0}}$ or $-r_e|_{T_{X^L_0}}$. Since $A_T$ has no $5$-primary part, the $5$-primary parts of $A_{T_{X^L_0}}$ and $A_U$ are paired by the gluing anti-isometry. The reflection $r_e$ is trivial on $A_{T_{X^L_0}}$, so $\iota^L_0$ acts on $U$ with the same overall sign. Hence $\iota^L_0|_T=r_e$ or $-r_e$.

As in Lemma~\ref{prop:m10}(2), $\iota^L_0$ fixes $h^2$ and acts trivially on $A_{\Lambda_0}$. Since $|A_S|=196$ is prime to $3$, the residual subgroup is $A_T[3]\cong A_{\Lambda_0}$. The reflection $r_e$ acts trivially on $A_T$, whereas $-r_e$ negates $A_T[3]$. Thus $\iota^L_0|_T=r_e$.

The cubics $X^L_t$ are smooth for $0\le t\le3\sqrt2/2$ by \cite[Table~1]{he2025cubicfourfoldsorder7automorphism}. Parallel transport along this real interval carries $\iota^L_0|_T$ to $\iota^L_+|_T$. Under this identification, $\iota^L_+|_T=r_e$ and $e\in T_{X^L_+}$ still has square $-2$.

Let $E$ be the real matrix in \cite[Proposition~3.16 and \S4.2]{he2025cubicfourfoldsorder7automorphism}. It preserves $X^L_+$, is non-symplectic, and satisfies $E^2\in L_2(7)$. We may take $g_+$ to be the automorphism induced by $E^3$. Its restriction to $T$ equals that of $E$, and it commutes with $\iota^L_+$. Thus $g_+$ preserves $\mathbb Qe$, the $-1$-eigenspace of $\iota^L_+|_T$. On $T$, its square is the identity and the period plane lies in its $-1$-eigenspace. Moreover, $g_+|_T\ne-\mathrm{id}_T$ because it fixes $A_T[3]$. Its $+1$-eigenspace in $T\otimes\mathbb R$ is therefore a positive line. Since $e^2=-2$, we have $g_+(e)=-e$.

Among the two integral actions computed in \cite[\S6.1.1]{FZ26}, only $-r_1$ has an anti-invariant lattice containing a vector of square $-2$. This assigns $-r_1$ to $X^L_+$. Both classes are realized by smooth cubics, and there are exactly two such cubics up to isomorphism by \cite[Theorem~1.2(5)]{he2025cubicfourfoldsorder7automorphism}. An isomorphism preserves the integral action, so the second class belongs to $X^L_-$.
\end{proof}

The required gluing subgroup of index $3$ on the $T$ side is unique. Both involutions act trivially on $A_T[3]$, so the extension criterion of \cite[Proposition~6.3]{he2025cubicfourfoldsorder7automorphism} applies.

\subsection{The $S_{3,3}$ and $A_{3,3}$ families}
There are two $S_{3,3}$ symplectic families $\mathscr{F}_{35}$ and $\mathscr{F}_{37}$ of generic non-symplectic index two \cite{KOIKE202512,KOIKE2026}, \cite{FWZ26}. Their generic automorphism groups are $S_{3,3}\times C_2$ and $N_{72}$, with Sylow $2$-subgroups $C_2^3$ and $D_8$, respectively. At index six, the groups are $S_{3,3}\times C_6$ and $N_{72}\times C_3$ for $\mathscr{F}_{36}$ and $\mathscr{F}_{38}$, respectively.

The second output of case 15 in \cite{FZ26} belongs to $\mathscr{F}_{36}$, and the first to $\mathscr{F}_{38}$; see \path{remark/catalogue/family_identification_script.md}. Taking the third power of these non-symplectic isometries gives the actions for $\mathscr{F}_{35}$ and $\mathscr{F}_{37}$, respectively.

The two $A_{3,3}$ symplectic families $\mathscr{F}_{52}$ and $\mathscr{F}_{55}$ have generic non-symplectic indices one and two, respectively \cite{KOIKE202512}, \cite{FWZ26}. The families $\mathscr{F}_{56}$ and $\mathscr{F}_{57}$ of non-symplectic index six have the same abstract group and dimension, and their moduli images lie in $\mathcal Z_{55}$. Their abstract pairs $(S,T)$ are isometric, but the multisets of traces below differ.

\begin{prop}\label{prop:a33-assignment}
Let $X_{56}$ and $X_{57}$ be general members of $\mathscr{F}_{56}$ and $\mathscr{F}_{57}$. For either cubic $X$, let $N=\Auts(X)$ and choose $f\in\Aut(X)$ whose image generates $\Aut(X)/N\cong C_6$.
\begin{enumerate}
    \item The multisets of pairs $(\tr(g\mid\Lambda_0),\tr(g^2\mid\Lambda_0))$, for $g\in fN$, are
\[
\begin{aligned}
X_{56}:\quad&\{(-1,7)^3,(2,-2)^6,(5,-11)^6,(5,7)^3\},\\
X_{57}:\quad&\{(2,-2)^{12},(5,-11)^6\},
\end{aligned}
\]
where superscripts indicate multiplicities. These multisets are independent of the choice of $f$ and distinguish the two families.
    \item The first and second saved actions of case 23 in \cite{FZ26} belong to $\mathscr{F}_{56}$ and $\mathscr{F}_{57}$, respectively.
\end{enumerate}
\end{prop}

\begin{proof}
For (1), in each case $|N|=18$ and $|\Aut(X)|=108$. Applying \eqref{eq:jacobian-character} and the residue formula to linear representatives of the elements $g\in fN$ gives the displayed multisets. The quotient $C_6$ has only two generators, so another choice gives $fN$ or $f^{-1}N$. Inversion is a bijection between these cosets. An orthogonal isometry and its inverse have the same trace, as do their squares, so both multisets are independent of this choice.

For (2), the discriminant kernel on $S$ for each saved action has order $18$. Adjoining the extra isometry gives a group of order $108$. For every product of the extra isometry with an element of the discriminant kernel, we compute both traces using $\Lambda_0\otimes\mathbb Q=(S\otimes\mathbb Q)\oplus(T\otimes\mathbb Q)$. The first output gives the multiset for $X_{56}$ and the second gives that for $X_{57}$, proving the assignment. See \path{remark/catalogue/a33_case23/result.md} and \path{remark/catalogue/a33_case23/case23_full_coset_comparison.mrdi} for the calculation and its saved result.
\end{proof}

\section{Action-maximal families}\label{sec:maximal}

An action is \emph{action-maximal} if its associated linear group $H$ is not conjugate to a proper subgroup of another group in the family list. There are 21 such actions, belonging to $\mathscr{F}_i$ for $i\in\{1,\ldots,8,13,14,18,20,24,41,44,96,127,152,154,155,156\}$. Every linear group in the list is conjugate to a subgroup of one of them \cite[Proposition~6.9 and Table~2]{FWZclass}. The first fifteen have $\rk S\geq15$; the other six have $\rk S<15$.

\subsection{Families with $\rk S\geq15$}\label{sec:known}

The families $\mathscr{F}_1,\ldots,\mathscr{F}_8$ have $\rk S=20$ and $\rk T=2$. Their lattice pairs and families are given by Laza--Zheng \cite[Theorem~1.8]{laza2022automorphisms} and Koike \cite{KOIKE202512,KOIKE2026}. For each pair we use its primitive embedding and a compatible extra isometry in $\widetilde O(\Lambda_0)$, and apply the primitive-extension and root tests of \cite[Algorithm~2]{FZ26}.

For $15\leq\rk S\leq19$, we use the actions computed in \cite{FZ26}. The two $M_9$ outputs represent one family, while the index-three output for $F_{21}$ is not saturated. Each classified symplectic lattice pair gives one class of embedded lattice data, except for $M_{10}$ and $S_{3,3}$, which give two \cite[Proposition~4.2]{FZ26}. We use the assignments of \S\ref{sec:two-pairs} for these exceptions.

For a symplectic family of generic non-symplectic index one, the extra isometry is the identity. For a symplectic family of positive dimension and generic non-symplectic index two, its restriction to $T$ is $-\mathrm{id}_T$ \cite{FWZ26}. We extend this action to an isometry in $\widetilde O(\Lambda_0)$, using \cite[Lemma~5.1]{FZ26} if a fitting extension does not lie in this group. This lemma assumes $\rk S<20$; for rank twenty we check the gluing directly.

This gives the actions for $\mathscr{F}_9$, $\mathscr{F}_{11}$, $\mathscr{F}_{16}$, and $\mathscr{F}_{19}$, whose $T$ has rank three. Other actions come from \cite{FZ26} or the restrictions of \S\ref{sec:restriction}.

For the $A_6$ family with $\rk S=19$ and generic non-symplectic index two and the $T_{48}$ family, we use the lattice pairs attached to their specified components in \cite{laza2022automorphisms} and \cite{KOIKE202512,KOIKE2026}. Their lattices $T$ have rank three, signature $(1,2)$, and discriminant symbols $4_5^{-1}3^{+1}5^{+1}$ and $2_7^{+1}8_{\mathrm{II}}^{-2}3^{+2}$, respectively.

The primitive embeddings and compatible extra isometries, with their group and root checks, are recorded in \path{remark/catalogue/ambient_completion/rank20_19/}. The file \path{provenance_and_status.md} in that directory lists the individual constructions.

\subsection{Families with $\rk S<15$}

There are six action-maximal families with $\rk(S)<15$.

Write $\omega=\zeta_3$. For each row, the listed matrices together with $\omega I_6$ generate $H_i$. The class $[g_i]\in G_i$ maps to a generator of $G_i/G_i^s$, and generates $G_i$ in the four cyclic cases.
\begin{center}
\small
\begin{tabular}{c l}
$i$ & generators other than $\omega I_6$\\
\hline
96 & $g_{96}=\diag(\zeta_8,-\zeta_4,-1,1,\omega,1)$,
       $b_{96}=(x_5\;x_6)$\\
127 & $g_{127}=\diag(\zeta_4,-1,1,1,1,-1)$,
       $s_{127}=\diag(1,1,1,1,-1,-1)$\\
152 & $g_{152}=\diag(\zeta_{16},-\zeta_8^3,\zeta_4,-\zeta_4,-1,1)$\\
154 & $g_{154}=\diag(-\zeta_8,-\zeta_4,\zeta_4,-1,1,\omega^2)$\\
155 & $g_{155}=\diag(\zeta_{32},-\zeta_{16}^7,\zeta_8,-\zeta_4,-1,1)$\\
156 & $g_{156}=\diag(\zeta_{16},-\zeta_8^3,\zeta_4,-1,1,\omega)$
\end{tabular}
\end{center}
Here $b_{96}$ interchanges $x_5$ and $x_6$. The spaces of invariant cubic forms are
{\small
\begin{align*}
W_{96}={}&\operatorname{span}\{x_1^2x_2,x_2^2x_3,x_3^2x_4,x_4^3,x_5^3+x_6^3\},\\
W_{127}={}&\operatorname{span}\{x_1^2x_2,x_2^2x_3,x_2^2x_4,x_2x_5x_6,x_3^3,x_3^2x_4,x_3x_4^2,x_4^3,x_3x_5^2,x_4x_5^2,x_3x_6^2,x_4x_6^2\},\\
W_{152}={}&\operatorname{span}\{x_1^2x_2,x_2^2x_3,x_3^2x_5,x_3x_4x_6,x_4^2x_5,x_5^2x_6,x_6^3\},\\
W_{154}={}&\operatorname{span}\{x_1^2x_2,x_2^2x_4,x_2x_3x_5,x_3^2x_4,x_4^2x_5,x_5^3,x_6^3\},\\
W_{155}={}&\operatorname{span}\{x_1^2x_2,x_2^2x_3,x_3^2x_4,x_4^2x_5,x_5^2x_6,x_6^3\},\\
W_{156}={}&\operatorname{span}\{x_1^2x_2,x_2^2x_3,x_3^2x_4,x_4^2x_5,x_5^3,x_6^3\}.
\end{align*}
}
The dimensions of these families are $0,4,1,1,0,0$, respectively.

Applying Proposition~\ref{prop:equation-character} to these matrices gives the characteristic polynomials in the table below. In the four cyclic cases the symplectic subgroup is trivial, so $T=\Lambda_0$. See \path{remark/low_rank/maximal_cases/maximal_family_character_result.md} for the calculation.

\begin{center}
\small
\begin{tabular}{c c c l}
$i$ & $G$ & $\rk T$ & $\det(t-g_i|_T)$\\
\hline
96 & $S_3\times C_{24}$ & 8 & $\Phi_{24}$\\
127 & $C_2\times C_4$ & 14 & $\Phi_1\Phi_2^3\Phi_4^5$\\
152 & $C_{16}$ & 22 & $\Phi_1^2\Phi_2^2\Phi_4\Phi_{16}^2$\\
154 & $C_{24}$ & 22 & $\Phi_3\Phi_{12}\Phi_{24}^2$\\
155 & $C_{32}$ & 22 & $\Phi_1\Phi_2\Phi_8\Phi_{32}$\\
156 & $C_{48}$ & 22 & $\Phi_3\Phi_{12}\Phi_{48}$
\end{tabular}
\end{center}
Since a finite-order action over $\mathbb Q$ is semisimple, its minimal polynomial is the product of the distinct cyclotomic factors in its characteristic polynomial. The multiplicities in the characteristic polynomial determine the ranks of the cyclotomic sublattices. By the Hodge--Riemann relations, the real subspace corresponding to $\chi(g_i),\chi(g_i)^{-1}$ has negative index two, and the other subspaces are positive definite.

\subsection{The $S_3\times C_{24}$ and cyclic actions}
For $S_3\times C_{24}$ we require an equivariant primitive extension, whereas in the cyclic cases $S=0$ and $T=\Lambda_0$. Algorithm~\ref{alg:five-searches} adapts \cite[Algorithms~2 and~3]{FZ26}, using the characteristic polynomials and signatures determined by the matrices above.

\begin{algorithm}[htbp]
\caption{Integral actions for $\mathscr{F}_{96}$, $\mathscr{F}_{152}$, $\mathscr{F}_{154}$, $\mathscr{F}_{155}$, and $\mathscr{F}_{156}$}\label{alg:five-searches}
\begin{algorithmic}[1]
\Require $(S,T)$, the order $m$ and exact characteristic and minimal polynomials of $f_T$, the signatures of its cyclotomic sublattices, and the family dimension; for a cyclic case, $S=0$ and $T=\Lambda_0$
\State Enumerate $f_T\in O(T)$ with the given order, polynomials, and signatures up to conjugacy; for $S_3\times C_{24}$, also take representatives under powers coprime to $24$
\ForAll{such $f_T$}
    \If{$S=0$}
        \State Let $\mathcal E=\{(\Lambda_0,f_T)\}$
    \Else
        \State Let $\mathcal E$ be the equivariant primitive extensions of $S\oplus T$ into $\Lambda_0$, with one fitting isometry $f$ for each class
    \EndIf
    \ForAll{$(L,f)\in\mathcal E$}
        \State Let $P=\ker\Phi_m(f_T)$ and $K=P^\perp_L$
        \If{$K$ contains a short or long root or the period dimension differs from the family dimension}
            \State \textbf{continue}
        \EndIf
        \If{$S=0$}
            \If{$f\notin\widetilde O(\Lambda_0)$ or $|\widetilde O(K)|\ne1$} \State \textbf{continue} \EndIf
        \Else
            \If{$|\widetilde O(K)|\ne|\widetilde O(S)|$} \State \textbf{continue} \EndIf
            \State Find a fitting isometry $f\in\widetilde O(L)$ in the extension class; if none exists, \textbf{continue}
            \State Check the group generated with $\widetilde O(S)$
        \EndIf
        \State Record $(L,f,S,T,P,K)$ if all tests pass
    \EndFor
\EndFor
\end{algorithmic}
\end{algorithm}

\begin{lem}\label{lem:cyclic-powers}
In each of the four cyclic cases, the classes recorded by Algorithm~\ref{alg:five-searches} form one orbit under powers coprime to the order. Thus they give one cyclic subgroup of $O(\Lambda_0)$ up to conjugacy.
\end{lem}

\begin{proof}
Let $m$ be the order of the cyclic group. The calculation records $\varphi(m)/2$ classes for each $m\in\{16,24,32,48\}$; see \path{remark/low_rank/maximal_cases/maximal_lattice_actions.md}. For any recorded isometry $g$, the lattice $P=\ker\Phi_m(g)$ has signature $(14,2)$ and positive definite orthogonal complement. Among the subspaces of $P\otimes\mathbb R$ corresponding to primitive eigenvalue pairs, exactly one has negative index two. Write its pair as $\{\zeta_m^a,\zeta_m^{-a}\}$.

For $(k,m)=1$, the power $g^k$ generates the same cyclic group. It has the same characteristic polynomial and cyclotomic sublattices as $g$, so it also passes the tests of Algorithm~\ref{alg:five-searches}. The corresponding pair of eigenvalues of $g^k$ is $\{\zeta_m^{ak},\zeta_m^{-ak}\}$. These give $\varphi(m)/2$ distinct pairs, which are invariant under integral conjugacy. The powers of any one recorded isometry therefore account for all the recorded classes.
\end{proof}

\begin{prop}\label{prop:five-searches}
Algorithm~\ref{alg:five-searches} determines the integral group actions of $\mathscr{F}_{96}$, $\mathscr{F}_{152}$, $\mathscr{F}_{154}$, $\mathscr{F}_{155}$, and $\mathscr{F}_{156}$ up to conjugacy in $O(\Lambda_0)$. In each case its outputs give one group action after forgetting the choice of generator of the non-symplectic quotient.
\end{prop}

\begin{proof}
Proposition~\ref{prop:equation-character} and the Hodge--Riemann relations give the characteristic polynomials and signatures of the geometric actions. These actions satisfy the tests of Algorithm~\ref{alg:five-searches}. The enumeration considers all such isometries and compatible primitive extensions, with isometries on $T$ taken up to coprime powers in the $S_3\times C_{24}$ case.

For $S_3\times C_{24}$, we have $\langle\widetilde O(S),f^k\rangle=\langle\widetilde O(S),f\rangle$ whenever $(k,24)=1$. The calculation gives one extension class with a fitting isometry in $\widetilde O(\Lambda_0)$ and the required group; see \path{remark/low_rank/maximal_cases/maximal_lattice_actions.md}. Here \cite[Lemma~5.1]{FZ26} applies because the symplectic family with group $S_3$ is connected and has generic non-symplectic index two. For fixed embedded lattices and action on $T$, two fitting isometries in $\widetilde O(\Lambda_0)$ differ by an element of $\widetilde O(S)$, so they generate the same group with the symplectic subgroup.

For the cyclic cases, Lemma~\ref{lem:cyclic-powers} gives one group action. The pair $\{\chi(g_i),\chi(g_i)^{-1}\}$ determines the conjugacy class of the recorded generator up to inversion.
\end{proof}

\subsection{The $C_2\times C_4$ family}
For $\mathscr{F}_{127}$, recall the matrices
\begin{align*}
g_{127}&=\diag(i,-1,1,1,1,-1),\\
s_{127}&=\diag(1,1,1,1,-1,-1).
\end{align*}
The involution $s_{127}$ generates the symplectic subgroup, and $g_{127}$ maps to a generator of the non-symplectic quotient $C_4$, with $\chi(g_{127})=i$. We construct commuting isometries $g,s\in\widetilde O(\Lambda_0)$ of orders four and two, with $g^2\ne s$, and take periods in the $i$-eigenspace of $g$.

The involution $g_{127}^2s_{127}=\diag(-1,1,1,1,-1,-1)$ has Marquand's $\phi_3$ type \cite[Proposition~2.12(3)]{MarquandInvolution}. Its anti-invariant lattice is $Q\cong U\oplus\langle2\rangle\oplus\langle-2\rangle\oplus E_8(2)$ \cite[Lemma~5.13]{MarquandInvolution}, which we use in the construction.

\begin{prop}\label{prop:c2c4}
There exist a lattice $L\cong\Lambda_0$ and commuting isometries $g,s\in\widetilde O(L)$ of orders four and two, respectively, with $g^2\ne s$, such that $S=L^{-s}\cong E_8(2)$, $T=L^s$, and
\[
\det(t-g|_T)=\Phi_1\Phi_2^3\Phi_4^5.
\]
The lattice $P=\ker\Phi_4(g|_T)$ has signature $(8,2)$, and $K_{\mathrm{per}}=P^\perp_L$ is positive definite and contains neither short nor long roots. Periods in the $i$-eigenspace of $g$ give a family of dimension four whose very general member has automorphism group $\langle g,s\rangle\cong C_4\times C_2$. This is the lattice action of $\mathscr{F}_{127}$.
\end{prop}

\begin{proof}
We choose a primitive $A_0\cong\langle4\rangle^{\oplus2}$ in $Q$ whose orthogonal complement $K_{\mathrm{aux}}$ has signature $(8,2)$ and determinant $2^{10}$. We realize $K_{\mathrm{aux}}$ as the trace lattice of a Hermitian lattice over $\mathbb Z[i]$, on which multiplication by $i$ gives an isometry of order four.

Using anti-isometries of discriminant subgroups as in \cite[Proposition~1.5.1]{Nikulin_1980}, equivariant gluing of $A_0\oplus K_{\mathrm{aux}}$ gives $Q$ with index $4$. Gluing $D_6(2)\oplus K_{\mathrm{aux}}$ with index $64$ gives $L'\cong U^{\oplus2}\oplus D_4^{\oplus3}$. Gluing $E_6(2)\oplus L'$ with index $64$ gives $L\cong\Lambda_0$ and commuting isometries $g,s$ with the stated properties. The Gram matrices, gluings, and checks of the actions and roots are recorded in \path{remark/low_rank/maximal_cases/family_127_g13_full_l_postcheck.md}.

Since $g,s\in\widetilde O(L)$ and $K_{\mathrm{per}}$ contains neither short nor long roots, \cite[Proposition~4.4]{FZ26} gives periods of smooth cubic fourfolds. The $i$-eigenspace of $g$ in $P_{\mathbb C}$ has dimension $5$, so these periods form an open subset of a complex ball of dimension four. The polarized Torelli theorem realizes $g$ and $s$ as automorphisms \cite[Proposition~2.4]{laza2022automorphisms}.

For a very general period, $P$ is the transcendental lattice and the symplectic group is $\widetilde O(K_{\mathrm{per}})$ \cite[Proposition~4.7(1)]{FZ26}. The orders of $O(K_{\mathrm{per}})$ and its image in $O(q_{K_{\mathrm{per}}})$ give
\[
|\widetilde O(K_{\mathrm{per}})|=\frac{8{,}847{,}360}{4{,}423{,}680}=2,
\]
so the symplectic group is generated by $s$. See \path{remark/low_rank/maximal_cases/family_127_g13_saturation_preflight.out}.

Let $n=[\Aut(X):\Auts(X)]$ for a very general member. Since $\chi(g)=i$, we have $4\mid n$. By \cite[Corollary~6.2]{laza2022automorphisms}, $\varphi(n)\mid\rk P=10$. Every multiple $n>4$ of four has $4\mid\varphi(n)$, so $n=4$ and $|\Aut(X)|=8$. As $g,s$ commute and $g^2\ne s$,
\[
\Aut(X)=\langle g,s\rangle\cong C_4\times C_2.
\]

In \cite[Table~1]{FWZclass}, $\mathscr{F}_{127}$ is the only family of dimension four with this group and $\rk S=8$. The action on cohomology is locally constant in this family, so the construction gives its lattice action.
\end{proof}

\section{Restrictions of lattice actions}\label{sec:restriction}

\subsection{Restriction along linear inclusions}

Suppose $u\in\GL(6,\CC)$ satisfies $u^{-1}H_i u\subset H_j$. After the coordinate change by $u$, every cubic fixed by $H_j$ is fixed by $H_i$ \cite[\S2A]{yu2020moduli}. Thus $\mathcal Z_j\subset\mathcal Z_i$ \cite[Proposition~6.9]{FWZclass}.

\begin{lem}\label{lem:restriction}
Suppose $u^{-1}H_i u\subset H_j$. If the integral action $\rho_j:G_j\to O(\Lambda_0)$ is known for a smooth member of $\mathscr{F}_j$, together with its action on $H^{3,1}$, then the action $\rho_i:G_i\to O(\Lambda_0)$ defined by $\rho_i([h])=\rho_j([u^{-1}hu])$ for $h\in H_i$ gives the integral action for a general member of $\mathscr{F}_i$, up to integral conjugacy.
\end{lem}

\begin{proof}
The coordinate change gives an inclusion of $\PP(W_j^{\mathrm{sm}})$ into $\PP(W_i^{\mathrm{sm}})$. The latter is a nonempty Zariski-open subset of $\PP(W_i)$ and is connected. Local markings make the $G_i$ action on integral cohomology locally constant up to conjugacy there \cite[Proposition~4.2]{yu2020moduli}. Thus the action at a member fixed by $H_j$ is the action at a general member of $\mathscr{F}_i$. By \cite[Theorem~6.10(1)]{FWZclass}, a general member of $\mathscr{F}_i$ has automorphism group $G_i$.
\end{proof}

The restriction gives $T_i=\Lambda_0^{G_i^s}$, $S_i=T_i^\perp$, their embeddings, and the extra isometry. To obtain the restriction from a chosen integral action of $G_j$, we find the subgroup corresponding to the linear inclusion by Algorithm~\ref{alg:subgroup-restriction}, using the isomorphism class of $G_i$, its non-symplectic index, and its traces on $\Lambda_0$.

\subsection{Restricting only the non-symplectic part}

Let $\Gamma$ be the image of $G_j$ in $O(\Lambda_0)$, let $N$ be its symplectic subgroup, and let $fN$ generate $\Gamma/N$. Subgroups containing $N$ are inverse images of subgroups of this cyclic quotient.

\begin{prop}\label{prop:same-part-restriction}
Suppose $\Gamma/N\cong C_m$ and $d\mid m$. The unique subgroup of $\Gamma$ containing $N$ with quotient of order $d$ is $\Gamma_d=\langle N,f^{m/d}\rangle$. If an inclusion of linear groups as in Lemma~\ref{lem:restriction} induces $\Gamma_d$, the corresponding family has the restricted integral action. Its embedded lattices $S$ and $T$ are not changed by restriction.
\end{prop}

We take $f^{m/d}$ on $\Lambda_0$ and restrict it to $S$ and $T$. If $d=1$, we use the identity as the extra isometry. The lattices $P_d=\ker\Phi_d(f^{m/d}|_T)$ and $K_d=(P_d)^\perp_{\Lambda_0}$ need not equal the original $P$ and $K$.

We apply this to the four cyclic action-maximal families; the powers are recorded in \path{remark/catalogue/cyclic_maximal_power_restrictions_result.md}. We also obtain $\mathscr{F}_{35}$ and $\mathscr{F}_{37}$ from $\mathscr{F}_{36}$ and $\mathscr{F}_{38}$, respectively, and $\mathscr{F}_{55}$ from either $\mathscr{F}_{56}$ or $\mathscr{F}_{57}$; see \path{remark/catalogue/ambient_completion/rank18_16/README.md}.

\subsection{Restricting both parts}

When the symplectic subgroup changes, we restrict the action of $\Gamma$, rather than only taking a power of its extra isometry. The symplectic subgroup $N$ of $\Gamma$ acts trivially on $T_j$. For an action of $G_j$ on a smooth cubic with $\rk S_j\geq13$, restriction to $S_j$ is faithful \cite[Lemma~4.3]{FZ26}, and the isomorphism class of $\Gamma$ can be computed from this restriction \cite[Lemma~4.5]{FZ26}. We check $\Gamma$ directly when $\rk S_j<13$. For each subgroup, we compare the multiset of $(\operatorname{ord}(g),\operatorname{tr}(g\mid\Lambda_0))$ over all its elements with that obtained from the action of $H_i$ on $\CC^6$ by \eqref{eq:jacobian-character} and the residue formula.

\begin{prop}\label{prop:changed-part-restriction}
Suppose $u^{-1}H_i u\subset H_j$, and let $\Gamma\subset O(\Lambda_0)$ be the image of the integral action of $G_j$, with symplectic subgroup $N$. Consider the subgroups of $\Gamma$ isomorphic to $G_i$ and having the non-symplectic index, multiset of pairs of orders and traces, $\rk S$, and period dimension of $\mathscr{F}_i$. If they form one conjugacy class in $\Gamma$, restriction to a representative $B$ gives the lattice action of $\mathscr{F}_i$, up to lattice-theoretic equivalence. Its lattices are $T_B=\Lambda_0^{B\cap N}$ and $S_B=T_B^\perp$.
\end{prop}

\begin{proof}
The linear inclusion gives a subgroup satisfying these conditions. Every other such subgroup is conjugate to it in $\Gamma$, hence in $O(\Lambda_0)$. Lemma~\ref{lem:restriction} shows that the restriction is the action of $\mathscr{F}_i$. Its symplectic subgroup is $B\cap N$, giving the stated lattices.
\end{proof}

Let $m=|\Gamma/N|$ and let $d\mid m$ be the non-symplectic index of $\mathscr{F}_i$. For $d>1$, every subgroup $B\leq\Gamma$ whose image in $\Gamma/N$ has order $d$ has the form $B=\langle N',b\rangle$, where $N'=B\cap N$ and $b\in Nf^{m/d}$. For $N'\leq N$ and $b$ in this coset,
\[
\langle N',b\rangle\cap N=N'
\quad\Longleftrightarrow\quad
bN'b^{-1}=N'\ \text{and}\ b^d\in N'.
\]
For $d=1$, we have $B=N'$. This proves the completeness of Algorithm~\ref{alg:subgroup-restriction} within the chosen action of $\Gamma$.

\begin{algorithm}[htbp]
\caption{Restriction to a smaller symplectic subgroup}\label{alg:subgroup-restriction}
\begin{algorithmic}[1]
\Require $\Gamma=\langle N,f\rangle\subset O(\Lambda_0)$, $m=|\Gamma/N|$, the non-symplectic index $d\mid m$, and the conditions for $\mathscr{F}_i$ in Proposition~\ref{prop:changed-part-restriction}
\ForAll{$N'\leq N$ up to conjugacy in $N$}
    \If{$d=1$}
        \State Let $\mathcal B_{N'}=\{N'\}$
    \Else
        \State Let $\mathcal B_{N'}=\{\langle N',b\rangle:b\in Nf^{m/d},\ bN'b^{-1}=N',\ b^d\in N'\}$
    \EndIf
    \ForAll{$B\in\mathcal B_{N'}$}
        \If{$B\cap N\ne N'$ or $B\not\cong G_i$ or its multiset of pairs of orders and traces differs from that computed from $H_i$}
            \State \textbf{continue}
        \EndIf
        \State Record $B$ up to conjugacy in $\Gamma$
    \EndFor
\EndFor
\ForAll{representatives $B$ of the recorded classes}
    \State Compute $T_B=\Lambda_0^{B\cap N}$, $S_B=T_B^\perp$, the action on $\Lambda_0$ of a lift in $B$ of a generator of $B/(B\cap N)$, $P_B$, and $K_B$; retain the required $\rk S_B$ and period dimension
\EndFor
\end{algorithmic}
\end{algorithm}

When Algorithm~\ref{alg:subgroup-restriction} gives several subgroup classes satisfying these conditions, we use linear actions preserving the equations of smooth cubics to determine the geometric restrictions. We compare such actions by linear conjugacy and connectedness of the smooth locus of invariant cubics.

For some families, more than one conjugacy class in $\Gamma$ satisfies the conditions. For $\mathscr{F}_{84}$, $\mathscr{F}_{85}$, and $\mathscr{F}_{107}$, a linear normalizer exchanges the two geometric subgroups of each type. Connectedness shows that their integral actions are conjugate. Linear conjugacies and connectedness give the same conclusion for $\mathscr{F}_{97}$ and $\mathscr{F}_{121}$; a block exchange on a smooth cubic does so for $\mathscr{F}_{124}$.

For $\mathscr{F}_{102}$ and $\mathscr{F}_{141}$, we restrict from $\mathscr{F}_{69}$ and $\mathscr{F}_{95}$, respectively. Algorithm~\ref{alg:subgroup-restriction} gives one matching class in each case; see \path{remark/catalogue/cross_symplectic_restrictions.md}.

For $\mathscr{F}_{143}$, the action of $\mathscr{F}_{41}$ is realized on a smooth cubic, and a representative of each of the three subgroup classes given by Algorithm~\ref{alg:subgroup-restriction} acts on that cubic. Choose a generator of one of these cyclic subgroups and a linear lift $g$ preserving the equation of the cubic. The induced actions of $g,g^2,g^3$ on $\Lambda_0$ have traces $(1,1,-2)$. By \eqref{eq:jacobian-character} and the residue formula, $g^6\ne I_6$; after inversion, $g^6=\omega I_6$. Smoothness forces each of the six eigenvalues $\pm\zeta_9^2\omega^j$, $j=0,1,2$, to occur once. The projective actions are therefore conjugate up to inversion. Connectedness of the smooth locus of invariant cubics and self-conjugacy of the family \cite[Proposition~6.6]{FWZclass} show that their integral actions are conjugate; see \path{remark/catalogue/cyclic_spectrum_equivalence_note.md}.

\subsection{An $A_4\times C_2$ family containing the Fermat cubic}

The family $\mathscr{F}_{105}$ has generic automorphism group $A_4\times C_2$ and dimension two. After a change of coordinates, $H_{97}\subset H_{105}$, hence $\mathcal Z_{105}\subset\mathcal Z_{97}$. The prescribed symplectic subgroup in both families is $C_2^2$ \cite{FWZclass}. We take their common embedded symplectic pair from the $\mathscr{F}_{97}$ action obtained from $\mathscr{F}_{74}$. The two subgroup classes in that restriction give equivalent pairs by linear conjugacy and connectedness. Here $\rk S=12$, $\rk T=10$, $\disc S=1024$, and $\disc T=3072$.

The group $A_4\times C_2$ has GAP ID $(24,13)$. The group $H_{105}$ is conjugate to a subgroup of the Fermat group. For the chosen $f$ mapping to a generator of the non-symplectic quotient, Proposition~\ref{prop:equation-character} applied to a linear representative gives
\begin{align*}
\det(t-f|_T)&=\Phi_1\Phi_2^3\Phi_6^3,\\
\det(t-f|_S)&=\Phi_1^3\Phi_2\Phi_3^3\Phi_6.
\end{align*}
The kernels of $\Phi_1(f_T)$, $\Phi_2(f_T)$, and $\Phi_6(f_T)$ have signatures $(1,0)$, $(3,0)$, and $(4,2)$, respectively. The conjugacy and characteristic polynomials are checked in \path{remark/catalogue/no105/verify_no104_no105_character.g}.

We use the embedded pair from $\mathscr{F}_{97}$ and adapt \cite[Algorithms~2 and~3]{FZ26} to these characteristic polynomials and signatures. Here we consider all fitting isometries to impose the characteristic polynomial on $S$.
\begin{enumerate}
\item Enumerate classes of isometries of order six in the genus of $T$ with the given characteristic and minimal polynomials and signatures of the cyclotomic sublattices. Keep only actions whose underlying lattice is isometric to the chosen $T$.
\item Since the primitive gluing has order $2^{10}$, require the action to fix the subgroup of $A_T$ of order three identified with $A_{\Lambda_0}$. This is necessary for an extension in $\widetilde O(\Lambda_0)$.
\item Enumerate all primitive extensions of each remaining action on $T$ to $\Lambda_0$ and all their fitting isometries. Keep the isometries in $\widetilde O(\Lambda_0)$ with the given polynomial on $S$ and with $K$ containing neither short nor long roots.
\end{enumerate}

\begin{prop}\label{prop:no105}
For the embedded symplectic pair of $\mathscr{F}_{97}$ and the characteristic polynomials and signatures above, exactly one class of primitive extensions has a fitting isometry $f\in\widetilde O(\Lambda_0)$ for which $K$ contains neither short nor long roots. It gives the lattice action of $\mathscr{F}_{105}$, up to lattice-theoretic equivalence, with $\rk P=6$ and $\rk K=16$.
\end{prop}

\begin{proof}
Let $N\cong C_2^2$ be the symplectic subgroup and $f$ a fitting isometry. Each nonidentity element of $N$ has anti-invariant rank $8$ on $S$ by \cite[Theorem~1.1(2)]{MarquandInvolution}. The given characteristic polynomial gives $\ord(f|_S)=6$ and anti-invariant rank $3$ for $f^3|_S$. Thus no $nf^k$ except the identity acts trivially on $S$: for $k\not\equiv0,3\pmod6$ the orders differ, and for $k\equiv3$ the anti-invariant ranks differ. The group $\langle N,f\rangle$ therefore acts faithfully on $S$, so its GAP ID can be computed from this restriction.

The enumeration gives exactly one extension class satisfying the polynomial and root conditions, with a fitting isometry in $\widetilde O(\Lambda_0)$. For this action, the period domain has dimension two, the generated group has GAP ID $(24,13)$, and $|\widetilde O(S)|=|\widetilde O(K)|=4$; see \path{remark/catalogue/no105/no105_result.md} and \path{remark/catalogue/no105/no105_extensions.mrdi}.

The Fermat cubic is a smooth member of $\mathscr{F}_{105}$ with these characteristic polynomials and embedded symplectic pair. Its action lies in $\widetilde O(\Lambda_0)$ and $K$ contains neither short nor long roots, so it belongs to the unique class. Lemma~\ref{lem:restriction} gives the same action for a general member of $\mathscr{F}_{105}$. The characteristic polynomial $\Phi_1\Phi_2\Phi_3\Phi_6^3$ on $T$ distinguishes $\mathscr{F}_{104}$ from this action.
\end{proof}

Then we come to the proof of the main theorem.

\begin{proof}[Proof of Theorem~\ref{thm:catalogue}]
\S\ref{sec:two-pairs} and \S\ref{sec:known} give the correspondence between the previously classified actions and the families. Propositions~\ref{prop:five-searches} and \ref{prop:c2c4} determine the six further action-maximal cases. Lemma~\ref{lem:restriction} and Propositions~\ref{prop:same-part-restriction} and \ref{prop:changed-part-restriction} give the restrictions. Proposition~\ref{prop:no105} treats $\mathscr{F}_{105}$. The source list in \path{remark/catalogue/ambient_completion/sources/family_156_sources_current.md} records one of these cases for each of the $156$ families of \cite{FWZclass}. The actions obtained by direct calculation satisfy the embedding and period conditions, while the restricted actions are realized by smooth cubic fourfolds. The construction of $A$ in \S\ref{sec:preliminaries} proves the last assertion.
\end{proof}

\section{Table of discriminant forms}\label{sec:data}
Table~\ref{tab:family-discriminants} lists the automorphism group $G$, $\rk S$, the index, the dimension, and the forms $q_S,q_T,q_P,q_K$ for each family. Here the index is the non-symplectic index $[G:G^s]$. The symbol $0$ denotes the trivial form. The MRDI file cited in the Introduction records the primitive embeddings and group actions.

\begingroup\footnotesize\setlength{\tabcolsep}{1.75pt}
\setlength{\LTleft}{0pt}\setlength{\LTright}{\fill}
\begin{longtable}{@{}l>{\raggedright\arraybackslash}p{72pt}lllllll@{}}
\caption{Discriminant forms of $S,T,P,K$ for the $156$ families}\label{tab:family-discriminants}\\
\hline
$i$ & $G$ & $\rk S$ & index & dim. & $q_S$ & $q_T$ & $q_P$ & $q_K$\\
\hline
\endfirsthead
\hline
$i$ & $G$ & $\rk S$ & index & dim. & $q_S$ & $q_T$ & $q_P$ & $q_K$\\
\hline
\endhead
1 & $C_3^5:S_6$ & 20 & 6 & 0 & $3^{+2}9^{-1}$ & $3^{+1}9^{+1}$ & $3^{+1}9^{+1}$ & $3^{+2}9^{-1}$\\
2 & $A_7$ & 20 & 1 & 0 & $3^{-1} 5^{+1} 7^{-1}$ & $3^{-2} 5^{+1} 7^{+1}$ & $3^{-2} 5^{+1} 7^{+1}$ & $3^{-1} 5^{+1} 7^{-1}$\\
3 & $S_7$ & 20 & 2 & 0 & $3^{-1} 5^{+1} 7^{-1}$ & $5^{+1} 7^{+1}$ & $5^{+1} 7^{+1}$ & $3^{-1} 5^{+1} 7^{-1}$\\
4 & $3^{1+4}:(C_4^2:C_2)$ & 20 & 4 & 0 & $2^{-2}_{2} 3^{-3}$ & $2^{+2}_{2} 3^{+2}$ & $2^{+2}_{2} 3^{+2}$ & $2^{-2}_{2} 3^{-3}$\\
5 & $M_{10}$ & 20 & 1 & 0 & $[2^{+1}4^{+1}]_{2} 3^{+1} 5^{+1}$ & $[2^{-1}4^{+1}]_{2} 3^{+2} 5^{+1}$ & $[2^{-1}4^{+1}]_{2} 3^{+2} 5^{+1}$ & $[2^{+1}4^{+1}]_{2} 3^{+1} 5^{+1}$\\
6 & $M_{10}$ & 20 & 1 & 0 & $[2^{+1}4^{+1}]_{2} 3^{+1} 5^{+1}$ & $[2^{-1}4^{+1}]_{2} 3^{+2} 5^{+1}$ & $[2^{-1}4^{+1}]_{2} 3^{+2} 5^{+1}$ & $[2^{+1}4^{+1}]_{2} 3^{+1} 5^{+1}$\\
7 & $L_2(11)\times C_3$ & 20 & 3 & 0 & $11^{+2}$ & $3^{-1} 11^{+2}$ & $3^{-1} 11^{+2}$ & $11^{+2}$\\
8 & $S_{3,5}\times C_3$ & 20 & 6 & 0 & $3^{-2} 5^{-2}$ & $3^{-1} 5^{-2}$ & $3^{-1} 5^{-2}$ & $3^{-2} 5^{-2}$\\
\hline
9 & $3^{1+4}:D_8$ & 19 & 2 & 1 & $2^{-1}_{3} 3^{-4}$ & $2^{-1}_{5} 3^{-3}$ & $2^{-1}_{5} 3^{-3}$ & $2^{-1}_{3} 3^{-4}$\\
10 & $A_6$ & 19 & 1 & 1 & $4^{-1}_{3} 3^{+2} 5^{+1}$ & $4^{-1}_{5} 3^{-3} 5^{+1}$ & $4^{-1}_{5} 3^{-3} 5^{+1}$ & $4^{-1}_{3} 3^{+2} 5^{+1}$\\
11 & $S_6$ & 19 & 2 & 1 & $4^{-1}_{3} 3^{+2} 5^{+1}$ & $4^{-1}_{5} 3^{+1} 5^{+1}$ & $4^{-1}_{5} 3^{+1} 5^{+1}$ & $4^{-1}_{3} 3^{+2} 5^{+1}$\\
12 & $L_2(7)$ & 19 & 1 & 1 & $4^{+1}_{7} 7^{+2}$ & $4^{+1}_{1} 3^{-1} 7^{+2}$ & $4^{+1}_{1} 3^{-1} 7^{+2}$ & $4^{+1}_{7} 7^{+2}$\\
13 & $L_2(7):C_2$ & 19 & 2 & 0 & $4^{+1}_{7} 7^{+2}$ & $4^{+1}_{1} 3^{-1} 7^{+2}$ & $[2^{+1}4^{+1}]_{0} 7^{-1}$ & $[2^{-1}4^{+1}]_{4} 3^{-1} 7^{+1}$\\
14 & $L_2(7):C_2$ & 19 & 2 & 0 & $4^{+1}_{7} 7^{+2}$ & $4^{+1}_{1} 3^{-1} 7^{+2}$ & $[2^{+1}4^{+1}]_{0} 7^{-1}$ & $[2^{+1}4^{+1}]_{0} 3^{-1} 7^{+1}$\\
15 & $S_5$ & 19 & 1 & 1 & $4^{-1}_{5} 3^{-1} 5^{-2}$ & $4^{-1}_{3} 3^{-2} 5^{-2}$ & $4^{-1}_{3} 3^{-2} 5^{-2}$ & $4^{-1}_{5} 3^{-1} 5^{-2}$\\
16 & $S_5\times C_2$ & 19 & 2 & 1 & $4^{-1}_{5} 3^{-1} 5^{-2}$ & $4^{-1}_{3} 5^{-2}$ & $4^{-1}_{3} 5^{-2}$ & $4^{-1}_{5} 3^{-1} 5^{-2}$\\
17 & $M_9$ & 19 & 1 & 1 & $2^{+3}_{5} 3^{+1}9^{+1}$ & $2^{+3}_{3} 3^{+2}9^{-1}$ & $2^{+3}_{3} 3^{+2}9^{-1}$ & $2^{+3}_{5} 3^{+1}9^{+1}$\\
18 & $M_9:C_3$ & 19 & 3 & 0 & $2^{+3}_{5} 3^{+1}9^{+1}$ & $2^{+3}_{3} 3^{+2}9^{-1}$ & $2^{-2}_{\mathrm{II}} 3^{-1}9^{-1}$ & $2^{-2}_{\mathrm{II}} 3^{-2}9^{+1}$\\
19 & $N_{72}\times C_2$ & 19 & 2 & 1 & $4^{+1}_{7} 3^{+2}9^{+1}$ & $4^{+1}_{1} 3^{+1}9^{-1}$ & $4^{+1}_{1} 3^{+1}9^{-1}$ & $4^{+1}_{7} 3^{+2}9^{+1}$\\
20 & $T_{48}$ & 19 & 1 & 1 & $2^{-1}_{5}8^{-2}_{\mathrm{II}} 3^{+1}$ & $2^{-1}_{3}8^{-2}_{\mathrm{II}} 3^{+2}$ & $2^{-1}_{3}8^{-2}_{\mathrm{II}} 3^{+2}$ & $2^{-1}_{5}8^{-2}_{\mathrm{II}} 3^{+1}$\\
\hline
21 & $3^{1+4}:C_2^2$ & 18 & 2 & 2 & $3^{-5}$ & $3^{+4}$ & $3^{+4}$ & $3^{-5}$\\
22 & $3^{1+4}:(C_4\times C_2)$ & 18 & 4 & 1 & $3^{-5}$ & $3^{+4}$ & $3^{+4}$ & $3^{-5}$\\
23 & $((C_3\times(C_3^3:\allowbreak C_3)):\allowbreak C_3):\allowbreak (C_2^2)$ & 18 & 6 & 1 & $3^{-5}$ & $3^{+4}$ & $3^{+4}$ & $3^{-5}$\\
24 & $((C_3 \times ((C_3^3) :\allowbreak  C_3)) :\allowbreak  C_3) :\allowbreak  (C_4 \times C_2)$ & 18 & 12 & 0 & $3^{-5}$ & $3^{+4}$ & $3^{+4}$ & $3^{-5}$\\
25 & $A_{4,3}$ & 18 & 1 & 2 & $4^{-2}_{\mathrm{II}} 3^{+3}$ & $4^{-2}_{\mathrm{II}} 3^{+4}$ & $4^{-2}_{\mathrm{II}} 3^{+4}$ & $4^{-2}_{\mathrm{II}} 3^{+3}$\\
26 & $A_{4,3}\times C_2$ & 18 & 2 & 1 & $4^{-2}_{\mathrm{II}} 3^{+3}$ & $4^{-2}_{\mathrm{II}} 3^{+4}$ & $8^{+1}_{1} 3^{-3}$ & $8^{+1}_{7} 3^{-4}$\\
27 & $S_{4,3}$ & 18 & 2 & 2 & $4^{-2}_{\mathrm{II}} 3^{+3}$ & $4^{-2}_{\mathrm{II}} 3^{-2}$ & $4^{-2}_{\mathrm{II}} 3^{-2}$ & $4^{-2}_{\mathrm{II}} 3^{+3}$\\
28 & $S_{4,3}\times C_3$ & 18 & 6 & 1 & $4^{-2}_{\mathrm{II}} 3^{+3}$ & $4^{-2}_{\mathrm{II}} 3^{-2}$ & $4^{-2}_{\mathrm{II}} 3^{-2}$ & $4^{-2}_{\mathrm{II}} 3^{+3}$\\
29 & $A_5$ & 18 & 1 & 2 & $2^{-2}_{\mathrm{II}} 3^{-1} 5^{-2}$ & $2^{-2}_{\mathrm{II}} 3^{-2} 5^{-2}$ & $2^{-2}_{\mathrm{II}} 3^{-2} 5^{-2}$ & $2^{-2}_{\mathrm{II}} 3^{-1} 5^{-2}$\\
30 & $A_5\times C_3$ & 18 & 3 & 1 & $2^{-2}_{\mathrm{II}} 3^{-1} 5^{-2}$ & $2^{-2}_{\mathrm{II}} 3^{-2} 5^{-2}$ & $2^{-2}_{\mathrm{II}} 3^{-2} 5^{-2}$ & $2^{-2}_{\mathrm{II}} 3^{-1} 5^{-2}$\\
31 & $S_5$ & 18 & 2 & 2 & $2^{-2}_{\mathrm{II}} 3^{-1} 5^{-2}$ & $2^{-2}_{\mathrm{II}} 5^{-2}$ & $2^{-2}_{\mathrm{II}} 5^{-2}$ & $2^{-2}_{\mathrm{II}} 3^{-1} 5^{-2}$\\
32 & $S_5\times C_3$ & 18 & 6 & 1 & $2^{-2}_{\mathrm{II}} 3^{-1} 5^{-2}$ & $2^{-2}_{\mathrm{II}} 5^{-2}$ & $2^{-2}_{\mathrm{II}} 5^{-2}$ & $2^{-2}_{\mathrm{II}} 3^{-1} 5^{-2}$\\
33 & $C_3^{2}.C_4$ & 18 & 1 & 2 & $2^{+2}_{6} 3^{+2}9^{+1}$ & $2^{-2}_{6} 3^{-3}9^{-1}$ & $2^{-2}_{6} 3^{-3}9^{-1}$ & $2^{+2}_{6} 3^{+2}9^{+1}$\\
34 & $N_{72}$ & 18 & 2 & 2 & $2^{+2}_{6} 3^{+2}9^{+1}$ & $2^{+2}_{2} 3^{+1}9^{-1}$ & $2^{+2}_{2} 3^{+1}9^{-1}$ & $2^{+2}_{6} 3^{+2}9^{+1}$\\
35 & $S_{3,3}\times C_2$ & 18 & 2 & 2 & $2^{-2}_{\mathrm{II}} 3^{-3}9^{+1}$ & $2^{-2}_{\mathrm{II}} 3^{+2}9^{-1}$ & $2^{-2}_{\mathrm{II}} 3^{+2}9^{-1}$ & $2^{-2}_{\mathrm{II}} 3^{-3}9^{+1}$\\
36 & $S_{3,3}\times C_6$ & 18 & 6 & 1 & $2^{-2}_{\mathrm{II}} 3^{-3}9^{+1}$ & $2^{-2}_{\mathrm{II}} 3^{+2}9^{-1}$ & $2^{-2}_{\mathrm{II}} 3^{+2}9^{-1}$ & $2^{-2}_{\mathrm{II}} 3^{-3}9^{+1}$\\
37 & $N_{72}$ & 18 & 2 & 2 & $2^{-2}_{\mathrm{II}} 3^{-3}9^{+1}$ & $2^{-2}_{\mathrm{II}} 3^{+2}9^{-1}$ & $2^{-2}_{\mathrm{II}} 3^{+2}9^{-1}$ & $2^{-2}_{\mathrm{II}} 3^{-3}9^{+1}$\\
38 & $N_{72}\times C_3$ & 18 & 6 & 1 & $2^{-2}_{\mathrm{II}} 3^{-3}9^{+1}$ & $2^{-2}_{\mathrm{II}} 3^{+2}9^{-1}$ & $2^{-2}_{\mathrm{II}} 3^{+2}9^{-1}$ & $2^{-2}_{\mathrm{II}} 3^{-3}9^{+1}$\\
39 & $F_{21}$ & 18 & 1 & 2 & $7^{-3}$ & $3^{-1} 7^{+3}$ & $3^{-1} 7^{+3}$ & $7^{-3}$\\
40 & $C_7:C_6$ & 18 & 2 & 1 & $7^{-3}$ & $3^{-1} 7^{+3}$ & $2^{+1}_{7} 7^{-2}$ & $2^{+1}_{1} 3^{-1} 7^{-2}$\\
41 & $(C_7:C_6)\times C_3$ & 18 & 6 & 0 & $7^{-3}$ & $3^{-1} 7^{+3}$ & $3^{+1} 7^{-2}$ & $3^{+2} 7^{-2}$\\
42 & $\mathrm{Hol}(5)$ & 18 & 1 & 2 & $2^{+2}_{6} 5^{+3}$ & $2^{+2}_{2} 3^{-1} 5^{+3}$ & $2^{+2}_{2} 3^{-1} 5^{+3}$ & $2^{+2}_{6} 5^{+3}$\\
43 & $\mathrm{QD}_{16}$ & 18 & 1 & 2 & $[2^{+1}4^{+1}]_{2}8^{+2}_{\mathrm{II}}$ & $[2^{-1}4^{+1}]_{2}8^{+2}_{\mathrm{II}} 3^{-1}$ & $[2^{-1}4^{+1}]_{2}8^{+2}_{\mathrm{II}} 3^{-1}$ & $[2^{+1}4^{+1}]_{2}8^{+2}_{\mathrm{II}}$\\
44 & $(C_8\times C_2): C_2$ & 18 & 2 & 1 & $[2^{+1}4^{+1}]_{2}8^{+2}_{\mathrm{II}}$ & $[2^{-1}4^{+1}]_{2}8^{+2}_{\mathrm{II}} 3^{-1}$ & $[4^{+1}8^{+2}]_{7}$ & $[4^{+1}8^{+2}]_{1} 3^{-1}$\\
\hline
45 & $S_4$ & 17 & 1 & 3 & $4^{+3}_{5} 3^{+2}$ & $4^{+3}_{3} 3^{-3}$ & $4^{+3}_{3} 3^{-3}$ & $4^{+3}_{5} 3^{+2}$\\
46 & $S_4\times C_2$ & 17 & 2 & 2 & $4^{+3}_{5} 3^{+2}$ & $4^{+3}_{3} 3^{-3}$ & $[4^{-1}8^{+1}]_{4} 3^{+2}$ & $[4^{-1}8^{+1}]_{4} 3^{-3}$\\
47 & $S_4\times C_2$ & 17 & 2 & 3 & $4^{+3}_{5} 3^{+2}$ & $4^{+3}_{3} 3^{+1}$ & $4^{+3}_{3} 3^{+1}$ & $4^{+3}_{5} 3^{+2}$\\
48 & $Q_8$ & 17 & 1 & 3 & $2^{+3}_{5}8^{-2}_{\mathrm{II}}$ & $2^{+3}_{3}8^{-2}_{\mathrm{II}} 3^{-1}$ & $2^{+3}_{3}8^{-2}_{\mathrm{II}} 3^{-1}$ & $2^{+3}_{5}8^{-2}_{\mathrm{II}}$\\
49 & $(C_4\times C_2):C_2$ & 17 & 2 & 2 & $2^{+3}_{5}8^{-2}_{\mathrm{II}}$ & $2^{+3}_{3}8^{-2}_{\mathrm{II}} 3^{-1}$ & $2^{+2}_{2}8^{+2}_{6}$ & $2^{-2}_{2}8^{+2}_{2} 3^{-1}$\\
50 & $\SL(2,3)$ & 17 & 3 & 1 & $2^{+3}_{5}8^{-2}_{\mathrm{II}}$ & $2^{+3}_{3}8^{-2}_{\mathrm{II}} 3^{-1}$ & $2^{-2}_{\mathrm{II}}8^{-2}_{\mathrm{II}} 3^{-2}$ & $2^{-2}_{\mathrm{II}}8^{-2}_{\mathrm{II}} 3^{-1}$\\
51 & $C_4^2: C_2$ & 17 & 4 & 1 & $2^{+3}_{5}8^{-2}_{\mathrm{II}}$ & $2^{+3}_{3}8^{-2}_{\mathrm{II}} 3^{-1}$ & $2^{+2}_{2}8^{+2}_{6}$ & $2^{-2}_{2}8^{+2}_{2} 3^{-1}$\\
\hline
52 & $A_{3,3}$ & 16 & 1 & 4 & $3^{+4}9^{+1}$ & $3^{-5}9^{-1}$ & $3^{-5}9^{-1}$ & $3^{+4}9^{+1}$\\
53 & $A_{3,3}\times C_2$ & 16 & 2 & 2 & $3^{+4}9^{+1}$ & $3^{-5}9^{-1}$ & $2^{+2}_{0} 3^{+4}$ & $2^{-2}_{4} 3^{-5}$\\
54 & $C_3^2:C_6$ & 16 & 3 & 1 & $3^{+4}9^{+1}$ & $3^{-5}9^{-1}$ & $3^{-2}9^{-2}$ & $3^{+3}9^{-2}$\\
55 & $S_{3,3}$ & 16 & 2 & 4 & $3^{+4}9^{+1}$ & $3^{+3}9^{-1}$ & $3^{+3}9^{-1}$ & $3^{+4}9^{+1}$\\
56 & $S_{3,3}\times C_3$ & 16 & 6 & 2 & $3^{+4}9^{+1}$ & $3^{+3}9^{-1}$ & $3^{+3}9^{-1}$ & $3^{+4}9^{+1}$\\
57 & $S_{3,3}\times C_3$ & 16 & 6 & 2 & $3^{+4}9^{+1}$ & $3^{+3}9^{-1}$ & $3^{+3}9^{-1}$ & $3^{+4}9^{+1}$\\
58 & $D_{12}$ & 16 & 1 & 4 & $2^{+4}_{\mathrm{II}} 3^{+4}$ & $2^{+4}_{\mathrm{II}} 3^{-5}$ & $2^{+4}_{\mathrm{II}} 3^{-5}$ & $2^{+4}_{\mathrm{II}} 3^{+4}$\\
59 & $(C_6\times C_2):C_2$ & 16 & 2 & 2 & $2^{+4}_{\mathrm{II}} 3^{+4}$ & $2^{+4}_{\mathrm{II}} 3^{-5}$ & $4^{-2}_{\mathrm{II}} 3^{+4}$ & $4^{-2}_{\mathrm{II}} 3^{-5}$\\
60 & $D_{12}\times C_2$ & 16 & 2 & 2 & $2^{+4}_{\mathrm{II}} 3^{+4}$ & $2^{+4}_{\mathrm{II}} 3^{-5}$ & $4^{+2}_{2} 3^{-3}$ & $4^{+2}_{6} 3^{-4}$\\
61 & $S_3\times C_6$ & 16 & 3 & 2 & $2^{+4}_{\mathrm{II}} 3^{+4}$ & $2^{+4}_{\mathrm{II}} 3^{-5}$ & $2^{+4}_{\mathrm{II}} 3^{-5}$ & $2^{+4}_{\mathrm{II}} 3^{+4}$\\
62 & $((C_6\times C_2):\allowbreak C_2)\times C_3$ & 16 & 6 & 1 & $2^{+4}_{\mathrm{II}} 3^{+4}$ & $2^{+4}_{\mathrm{II}} 3^{-5}$ & $4^{-2}_{\mathrm{II}} 3^{+4}$ & $4^{-2}_{\mathrm{II}} 3^{-5}$\\
63 & $D_{12}\times C_2$ & 16 & 2 & 4 & $2^{+4}_{\mathrm{II}} 3^{+4}$ & $2^{+4}_{\mathrm{II}} 3^{+3}$ & $2^{+4}_{\mathrm{II}} 3^{+3}$ & $2^{+4}_{\mathrm{II}} 3^{+4}$\\
64 & $S_3\times C_6\times C_2$ & 16 & 6 & 2 & $2^{+4}_{\mathrm{II}} 3^{+4}$ & $2^{+4}_{\mathrm{II}} 3^{+3}$ & $2^{+4}_{\mathrm{II}} 3^{+3}$ & $2^{+4}_{\mathrm{II}} 3^{+4}$\\
65 & $A_4$ & 16 & 1 & 4 & $2^{-2}_{\mathrm{II}}4^{-2}_{\mathrm{II}} 3^{+2}$ & $2^{-2}_{\mathrm{II}}4^{-2}_{\mathrm{II}} 3^{-3}$ & $2^{-2}_{\mathrm{II}}4^{-2}_{\mathrm{II}} 3^{-3}$ & $2^{-2}_{\mathrm{II}}4^{-2}_{\mathrm{II}} 3^{+2}$\\
66 & $A_4\times C_2$ & 16 & 2 & 3 & $2^{-2}_{\mathrm{II}}4^{-2}_{\mathrm{II}} 3^{+2}$ & $2^{-2}_{\mathrm{II}}4^{-2}_{\mathrm{II}} 3^{-3}$ & $2^{-2}_{\mathrm{II}}8^{+1}_{1} 3^{+2}$ & $2^{-2}_{\mathrm{II}}8^{+1}_{7} 3^{-3}$\\
67 & $A_4\times C_3$ & 16 & 3 & 2 & $2^{-2}_{\mathrm{II}}4^{-2}_{\mathrm{II}} 3^{+2}$ & $2^{-2}_{\mathrm{II}}4^{-2}_{\mathrm{II}} 3^{-3}$ & $2^{-2}_{\mathrm{II}}4^{-2}_{\mathrm{II}} 3^{-3}$ & $2^{-2}_{\mathrm{II}}4^{-2}_{\mathrm{II}} 3^{+2}$\\
68 & $S_4$ & 16 & 2 & 4 & $2^{-2}_{\mathrm{II}}4^{-2}_{\mathrm{II}} 3^{+2}$ & $2^{-2}_{\mathrm{II}}4^{-2}_{\mathrm{II}} 3^{+1}$ & $2^{-2}_{\mathrm{II}}4^{-2}_{\mathrm{II}} 3^{+1}$ & $2^{-2}_{\mathrm{II}}4^{-2}_{\mathrm{II}} 3^{+2}$\\
69 & $S_4\times C_3$ & 16 & 6 & 2 & $2^{-2}_{\mathrm{II}}4^{-2}_{\mathrm{II}} 3^{+2}$ & $2^{-2}_{\mathrm{II}}4^{-2}_{\mathrm{II}} 3^{+1}$ & $2^{-2}_{\mathrm{II}}4^{-2}_{\mathrm{II}} 3^{+1}$ & $2^{-2}_{\mathrm{II}}4^{-2}_{\mathrm{II}} 3^{+2}$\\
70 & $D_{10}$ & 16 & 1 & 4 & $5^{+4}$ & $3^{-1} 5^{+4}$ & $3^{-1} 5^{+4}$ & $5^{+4}$\\
71 & $D_{20}$ & 16 & 2 & 2 & $5^{+4}$ & $3^{-1} 5^{+4}$ & $2^{+2}_{\mathrm{II}} 5^{-3}$ & $2^{+2}_{\mathrm{II}} 3^{-1} 5^{-3}$\\
72 & $D_{10}\times C_3$ & 16 & 3 & 2 & $5^{+4}$ & $3^{-1} 5^{+4}$ & $3^{-1} 5^{+4}$ & $5^{+4}$\\
\hline
73 & $D_8$ & 15 & 1 & 5 & $4^{+5}_{7}$ & $4^{+5}_{1} 3^{-1}$ & $4^{+5}_{1} 3^{-1}$ & $4^{+5}_{7}$\\
74 & $D_8\times C_2$ & 15 & 2 & 4 & $4^{+5}_{7}$ & $4^{+5}_{1} 3^{-1}$ & $[4^{-3}8^{+1}]_{2}$ & $[4^{-3}8^{+1}]_{6} 3^{-1}$\\
75 & $(C_4\times C_2):C_2$ & 15 & 2 & 3 & $4^{+5}_{7}$ & $4^{+5}_{1} 3^{-1}$ & $4^{+5}_{1}$ & $4^{+5}_{7} 3^{-1}$\\
76 & $D_{16}$ & 15 & 2 & 2 & $4^{+5}_{7}$ & $4^{+5}_{1} 3^{-1}$ & $[2^{+1}4^{+1}]_{0}8^{+2}_{\mathrm{II}}$ & $[2^{+1}4^{+1}]_{0}8^{+2}_{\mathrm{II}} 3^{-1}$\\
\hline
77 & $C_{4}$ & 14 & 1 & 6 & $2^{+2}_{6}4^{+4}_{\mathrm{II}}$ & $2^{-2}_{6}4^{+4}_{\mathrm{II}} 3^{-1}$ & $2^{-2}_{6}4^{+4}_{\mathrm{II}} 3^{-1}$ & $2^{+2}_{6}4^{+4}_{\mathrm{II}}$\\
78 & $D_8$ & 14 & 2 & 5 & $2^{+2}_{6}4^{+4}_{\mathrm{II}}$ & $2^{-2}_{6}4^{+4}_{\mathrm{II}} 3^{-1}$ & $2^{+2}_{2}4^{+2}_{\mathrm{II}}8^{+1}_{1}$ & $2^{-2}_{2}4^{+2}_{\mathrm{II}}8^{+1}_{7} 3^{-1}$\\
79 & $C_{4} \times C_{2}$ & 14 & 2 & 4 & $2^{+2}_{6}4^{+4}_{\mathrm{II}}$ & $2^{-2}_{6}4^{+4}_{\mathrm{II}} 3^{-1}$ & $2^{+2}_{2}4^{+4}_{\mathrm{II}}$ & $2^{-2}_{2}4^{+4}_{\mathrm{II}} 3^{-1}$\\
80 & $D_8$ & 14 & 2 & 3 & $2^{+2}_{6}4^{+4}_{\mathrm{II}}$ & $2^{-2}_{6}4^{+4}_{\mathrm{II}} 3^{-1}$ & $2^{+3}_{1}8^{+2}_{\mathrm{II}}$ & $2^{+3}_{7}8^{+2}_{\mathrm{II}} 3^{-1}$\\
81 & $C_{4} \times C_{4}$ & 14 & 4 & 2 & $2^{+2}_{6}4^{+4}_{\mathrm{II}}$ & $2^{-2}_{6}4^{+4}_{\mathrm{II}} 3^{-1}$ & $2^{+2}_{2}4^{+4}_{\mathrm{II}}$ & $2^{-2}_{2}4^{+4}_{\mathrm{II}} 3^{-1}$\\
82 & $C_8:C_2$ & 14 & 4 & 2 & $2^{+2}_{6}4^{+4}_{\mathrm{II}}$ & $2^{-2}_{6}4^{+4}_{\mathrm{II}} 3^{-1}$ & $2^{+2}_{2}4^{+4}_{\mathrm{II}}$ & $2^{-2}_{2}4^{+4}_{\mathrm{II}} 3^{-1}$\\
83 & $S_3$ & 14 & 1 & 6 & $2^{-2}_{\mathrm{II}} 3^{-5}$ & $2^{-2}_{\mathrm{II}} 3^{-6}$ & $2^{-2}_{\mathrm{II}} 3^{-6}$ & $2^{-2}_{\mathrm{II}} 3^{-5}$\\
84 & $D_{12}$ & 14 & 2 & 3 & $2^{-2}_{\mathrm{II}} 3^{-5}$ & $2^{-2}_{\mathrm{II}} 3^{-6}$ & $[2^{-2}4^{+1}]_{7} 3^{-3}$ & $[2^{-2}4^{+1}]_{1} 3^{-4}$\\
85 & $D_{12}$ & 14 & 2 & 3 & $2^{-2}_{\mathrm{II}} 3^{-5}$ & $2^{-2}_{\mathrm{II}} 3^{-6}$ & $2^{-2}_{\mathrm{II}}4^{+1}_{1} 3^{-4}$ & $2^{-2}_{\mathrm{II}}4^{+1}_{7} 3^{+5}$\\
86 & $S_3\times C_3$ & 14 & 3 & 3 & $2^{-2}_{\mathrm{II}} 3^{-5}$ & $2^{-2}_{\mathrm{II}} 3^{-6}$ & $2^{-2}_{\mathrm{II}} 3^{-6}$ & $2^{-2}_{\mathrm{II}} 3^{-5}$\\
87 & $S_3\times C_3$ & 14 & 3 & 2 & $2^{-2}_{\mathrm{II}} 3^{-5}$ & $2^{-2}_{\mathrm{II}} 3^{-6}$ & $2^{-2}_{\mathrm{II}} 3^{-3}9^{+1}$ & $2^{-2}_{\mathrm{II}} 3^{-4}9^{-1}$\\
88 & $D_{12}$ & 14 & 2 & 6 & $2^{-2}_{\mathrm{II}} 3^{-5}$ & $2^{-2}_{\mathrm{II}} 3^{+4}$ & $2^{-2}_{\mathrm{II}} 3^{+4}$ & $2^{-2}_{\mathrm{II}} 3^{-5}$\\
89 & $S_3\times C_4$ & 14 & 4 & 3 & $2^{-2}_{\mathrm{II}} 3^{-5}$ & $2^{-2}_{\mathrm{II}} 3^{+4}$ & $2^{-2}_{\mathrm{II}} 3^{+4}$ & $2^{-2}_{\mathrm{II}} 3^{-5}$\\
90 & $S_3\times C_6$ & 14 & 6 & 3 & $2^{-2}_{\mathrm{II}} 3^{-5}$ & $2^{-2}_{\mathrm{II}} 3^{+4}$ & $2^{-2}_{\mathrm{II}} 3^{+4}$ & $2^{-2}_{\mathrm{II}} 3^{-5}$\\
91 & $S_3\times C_6$ & 14 & 6 & 3 & $2^{-2}_{\mathrm{II}} 3^{-5}$ & $2^{-2}_{\mathrm{II}} 3^{+4}$ & $2^{-2}_{\mathrm{II}} 3^{+4}$ & $2^{-2}_{\mathrm{II}} 3^{-5}$\\
92 & $S_3\times C_6$ & 14 & 6 & 2 & $2^{-2}_{\mathrm{II}} 3^{-5}$ & $2^{-2}_{\mathrm{II}} 3^{+4}$ & $2^{-2}_{\mathrm{II}} 3^{-3}9^{+1}$ & $2^{-2}_{\mathrm{II}} 3^{-4}9^{-1}$\\
93 & $S_3\times C_8$ & 14 & 8 & 1 & $2^{-2}_{\mathrm{II}} 3^{-5}$ & $2^{-2}_{\mathrm{II}} 3^{+4}$ & $2^{-2}_{\mathrm{II}} 3^{+4}$ & $2^{-2}_{\mathrm{II}} 3^{-5}$\\
94 & $S_3\times C_{12}$ & 14 & 12 & 1 & $2^{-2}_{\mathrm{II}} 3^{-5}$ & $2^{-2}_{\mathrm{II}} 3^{+4}$ & $2^{-2}_{\mathrm{II}} 3^{+4}$ & $2^{-2}_{\mathrm{II}} 3^{-5}$\\
95 & $S_3\times C_{12}$ & 14 & 12 & 1 & $2^{-2}_{\mathrm{II}} 3^{-5}$ & $2^{-2}_{\mathrm{II}} 3^{+4}$ & $2^{-2}_{\mathrm{II}} 3^{+4}$ & $2^{-2}_{\mathrm{II}} 3^{-5}$\\
96 & $S_3\times C_{24}$ & 14 & 24 & 0 & $2^{-2}_{\mathrm{II}} 3^{-5}$ & $2^{-2}_{\mathrm{II}} 3^{+4}$ & $2^{-2}_{\mathrm{II}} 3^{+4}$ & $2^{-2}_{\mathrm{II}} 3^{-5}$\\
\hline
97 & $C_2^2$ & 12 & 1 & 8 & $2^{-6}_{\mathrm{II}}4^{-2}_{\mathrm{II}}$ & $2^{-6}_{\mathrm{II}}4^{-2}_{\mathrm{II}} 3^{-1}$ & $2^{-6}_{\mathrm{II}}4^{-2}_{\mathrm{II}} 3^{-1}$ & $2^{-6}_{\mathrm{II}}4^{-2}_{\mathrm{II}}$\\
98 & $C_2^3$ & 12 & 2 & 7 & $2^{-6}_{\mathrm{II}}4^{-2}_{\mathrm{II}}$ & $2^{-6}_{\mathrm{II}}4^{-2}_{\mathrm{II}} 3^{-1}$ & $2^{-6}_{\mathrm{II}}8^{+1}_{1}$ & $2^{-6}_{\mathrm{II}}8^{+1}_{7} 3^{-1}$\\
99 & $D_8$ & 12 & 2 & 6 & $2^{-6}_{\mathrm{II}}4^{-2}_{\mathrm{II}}$ & $2^{-6}_{\mathrm{II}}4^{-2}_{\mathrm{II}} 3^{-1}$ & $2^{-2}_{\mathrm{II}}4^{+4}_{\mathrm{II}}$ & $2^{-2}_{\mathrm{II}}4^{+4}_{\mathrm{II}} 3^{-1}$\\
100 & $C_2^3$ & 12 & 2 & 5 & $2^{-6}_{\mathrm{II}}4^{-2}_{\mathrm{II}}$ & $2^{-6}_{\mathrm{II}}4^{-2}_{\mathrm{II}} 3^{-1}$ & $4^{-5}_{3}$ & $4^{-5}_{5} 3^{-1}$\\
101 & $C_{6} \times C_{2}$ & 12 & 3 & 4 & $2^{-6}_{\mathrm{II}}4^{-2}_{\mathrm{II}}$ & $2^{-6}_{\mathrm{II}}4^{-2}_{\mathrm{II}} 3^{-1}$ & $2^{-6}_{\mathrm{II}}4^{-2}_{\mathrm{II}} 3^{-1}$ & $2^{-6}_{\mathrm{II}}4^{-2}_{\mathrm{II}}$\\
102 & $A_4$ & 12 & 3 & 3 & $2^{-6}_{\mathrm{II}}4^{-2}_{\mathrm{II}}$ & $2^{-6}_{\mathrm{II}}4^{-2}_{\mathrm{II}} 3^{-1}$ & $2^{+4}_{\mathrm{II}}4^{-2}_{\mathrm{II}} 3^{+2}$ & $2^{+4}_{\mathrm{II}}4^{-2}_{\mathrm{II}} 3^{+1}$\\
103 & $D_8\times C_3$ & 12 & 6 & 3 & $2^{-6}_{\mathrm{II}}4^{-2}_{\mathrm{II}}$ & $2^{-6}_{\mathrm{II}}4^{-2}_{\mathrm{II}} 3^{-1}$ & $2^{-2}_{\mathrm{II}}4^{+4}_{\mathrm{II}}$ & $2^{-2}_{\mathrm{II}}4^{+4}_{\mathrm{II}} 3^{-1}$\\
104 & $A_4\times C_2$ & 12 & 6 & 2 & $2^{-6}_{\mathrm{II}}4^{-2}_{\mathrm{II}}$ & $2^{-6}_{\mathrm{II}}4^{-2}_{\mathrm{II}} 3^{-1}$ & $4^{+4}_{\mathrm{II}} 3^{-1}$ & $4^{+4}_{\mathrm{II}} 3^{-2}$\\
105 & $A_4\times C_2$ & 12 & 6 & 2 & $2^{-6}_{\mathrm{II}}4^{-2}_{\mathrm{II}}$ & $2^{-6}_{\mathrm{II}}4^{-2}_{\mathrm{II}} 3^{-1}$ & $2^{+4}_{\mathrm{II}} 3^{+3}$ & $2^{+4}_{\mathrm{II}} 3^{+4}$\\
106 & $C_{3}$ & 12 & 1 & 8 & $3^{+6}$ & $3^{-7}$ & $3^{-7}$ & $3^{+6}$\\
107 & $S_3$ & 12 & 2 & 4 & $3^{+6}$ & $3^{-7}$ & $2^{+4}_{6} 3^{-4}$ & $2^{-4}_{6} 3^{+5}$\\
108 & $C_{6}$ & 12 & 2 & 4 & $3^{+6}$ & $3^{-7}$ & $2^{+4}_{\mathrm{II}} 3^{-5}$ & $2^{+4}_{\mathrm{II}} 3^{-6}$\\
109 & $C_{6}$ & 12 & 2 & 4 & $3^{+6}$ & $3^{-7}$ & $2^{-4}_{0} 3^{-3}$ & $2^{+4}_{4} 3^{-4}$\\
110 & $C_3^2$ & 12 & 3 & 4 & $3^{+6}$ & $3^{-7}$ & $3^{-7}$ & $3^{+6}$\\
111 & $C_3^2$ & 12 & 3 & 3 & $3^{+6}$ & $3^{-7}$ & $3^{-4}9^{+1}$ & $3^{+5}9^{-1}$\\
112 & $C_3^2$ & 12 & 3 & 2 & $3^{+6}$ & $3^{-7}$ & $3^{+3}9^{+2}$ & $3^{+4}9^{+2}$\\
113 & $C_{6} \times C_{3}$ & 12 & 6 & 2 & $3^{+6}$ & $3^{-7}$ & $2^{+4}_{\mathrm{II}} 3^{-5}$ & $2^{+4}_{\mathrm{II}} 3^{-6}$\\
114 & $C_{6} \times C_{3}$ & 12 & 6 & 1 & $3^{+6}$ & $3^{-7}$ & $2^{-2}_{\mathrm{II}} 3^{+2}9^{+1}$ & $2^{-2}_{\mathrm{II}} 3^{-3}9^{-1}$\\
115 & $S_3$ & 12 & 2 & 8 & $3^{+6}$ & $3^{+5}$ & $3^{+5}$ & $3^{+6}$\\
116 & $S_3\times C_3$ & 12 & 6 & 4 & $3^{+6}$ & $3^{+5}$ & $3^{+5}$ & $3^{+6}$\\
117 & $S_3\times C_3$ & 12 & 6 & 4 & $3^{+6}$ & $3^{+5}$ & $3^{+5}$ & $3^{+6}$\\
118 & $S_3\times C_3$ & 12 & 6 & 3 & $3^{+6}$ & $3^{+5}$ & $3^{-4}9^{+1}$ & $3^{+5}9^{-1}$\\
\hline
119 & $C_{2}$ & 8 & 1 & 12 & $2^{+8}_{\mathrm{II}}$ & $2^{+8}_{\mathrm{II}} 3^{-1}$ & $2^{+8}_{\mathrm{II}} 3^{-1}$ & $2^{+8}_{\mathrm{II}}$\\
120 & $C_2^2$ & 8 & 2 & 10 & $2^{+8}_{\mathrm{II}}$ & $2^{+8}_{\mathrm{II}} 3^{-1}$ & $2^{+4}_{\mathrm{II}}4^{-2}_{\mathrm{II}}$ & $2^{+4}_{\mathrm{II}}4^{-2}_{\mathrm{II}} 3^{-1}$\\
121 & $C_2^2$ & 8 & 2 & 8 & $2^{+8}_{\mathrm{II}}$ & $2^{+8}_{\mathrm{II}} 3^{-1}$ & $2^{-6}_{\mathrm{II}}4^{+2}_{2}$ & $2^{-6}_{\mathrm{II}}4^{+2}_{6} 3^{-1}$\\
122 & $C_2^2$ & 8 & 2 & 6 & $2^{+8}_{\mathrm{II}}$ & $2^{+8}_{\mathrm{II}} 3^{-1}$ & $[2^{+2}4^{+4}]_{4}$ & $[2^{-2}4^{+4}]_{0} 3^{-1}$\\
123 & $C_{6}$ & 8 & 3 & 6 & $2^{+8}_{\mathrm{II}}$ & $2^{+8}_{\mathrm{II}} 3^{-1}$ & $2^{+8}_{\mathrm{II}} 3^{-1}$ & $2^{+8}_{\mathrm{II}}$\\
124 & $C_{4} \times C_{2}$ & 8 & 4 & 5 & $2^{+8}_{\mathrm{II}}$ & $2^{+8}_{\mathrm{II}} 3^{-1}$ & $2^{+4}_{\mathrm{II}}4^{-2}_{\mathrm{II}}$ & $2^{+4}_{\mathrm{II}}4^{-2}_{\mathrm{II}} 3^{-1}$\\
125 & $C_{6} \times C_{2}$ & 8 & 6 & 5 & $2^{+8}_{\mathrm{II}}$ & $2^{+8}_{\mathrm{II}} 3^{-1}$ & $2^{+4}_{\mathrm{II}}4^{-2}_{\mathrm{II}}$ & $2^{+4}_{\mathrm{II}}4^{-2}_{\mathrm{II}} 3^{-1}$\\
126 & $C_{6}$ & 8 & 3 & 4 & $2^{+8}_{\mathrm{II}}$ & $2^{+8}_{\mathrm{II}} 3^{-1}$ & $2^{-6}_{\mathrm{II}} 3^{-5}$ & $2^{-6}_{\mathrm{II}} 3^{+4}$\\
127 & $C_{4} \times C_{2}$ & 8 & 4 & 4 & $2^{+8}_{\mathrm{II}}$ & $2^{+8}_{\mathrm{II}} 3^{-1}$ & $2^{-6}_{\mathrm{II}}4^{+2}_{2}$ & $2^{-6}_{\mathrm{II}}4^{+2}_{6} 3^{-1}$\\
128 & $C_{6} \times C_{2}$ & 8 & 6 & 3 & $2^{+8}_{\mathrm{II}}$ & $2^{+8}_{\mathrm{II}} 3^{-1}$ & $2^{-2}_{\mathrm{II}}4^{-2}_{\mathrm{II}} 3^{+4}$ & $2^{-2}_{\mathrm{II}}4^{-2}_{\mathrm{II}} 3^{-5}$\\
129 & $C_{8} \times C_{2}$ & 8 & 8 & 2 & $2^{+8}_{\mathrm{II}}$ & $2^{+8}_{\mathrm{II}} 3^{-1}$ & $2^{+4}_{\mathrm{II}}4^{-2}_{\mathrm{II}}$ & $2^{+4}_{\mathrm{II}}4^{-2}_{\mathrm{II}} 3^{-1}$\\
130 & $C_{12} \times C_{2}$ & 8 & 12 & 2 & $2^{+8}_{\mathrm{II}}$ & $2^{+8}_{\mathrm{II}} 3^{-1}$ & $2^{+4}_{\mathrm{II}}4^{-2}_{\mathrm{II}}$ & $2^{+4}_{\mathrm{II}}4^{-2}_{\mathrm{II}} 3^{-1}$\\
131 & $C_{12} \times C_{2}$ & 8 & 12 & 1 & $2^{+8}_{\mathrm{II}}$ & $2^{+8}_{\mathrm{II}} 3^{-1}$ & $2^{-2}_{\mathrm{II}}4^{-2}_{\mathrm{II}} 3^{+4}$ & $2^{-2}_{\mathrm{II}}4^{-2}_{\mathrm{II}} 3^{-5}$\\
\hline
132 & $1$ & 0 & 1 & 20 & $0$ & $3^{-1}$ & $3^{-1}$ & $0$\\
133 & $C_{2}$ & 0 & 2 & 14 & $0$ & $3^{-1}$ & $2^{-6}_{\mathrm{II}}$ & $2^{-6}_{\mathrm{II}} 3^{-1}$\\
134 & $C_{2}$ & 0 & 2 & 10 & $0$ & $3^{-1}$ & $2^{+10}_{0}$ & $2^{-10}_{4} 3^{-1}$\\
135 & $C_{3}$ & 0 & 3 & 10 & $0$ & $3^{-1}$ & $3^{-1}$ & $0$\\
136 & $C_{3}$ & 0 & 3 & 7 & $0$ & $3^{-1}$ & $3^{+6}$ & $3^{+5}$\\
137 & $C_{4}$ & 0 & 4 & 7 & $0$ & $3^{-1}$ & $2^{-6}_{\mathrm{II}}$ & $2^{-6}_{\mathrm{II}} 3^{-1}$\\
138 & $C_{4}$ & 0 & 4 & 7 & $0$ & $3^{-1}$ & $2^{-6}_{\mathrm{II}}$ & $2^{-6}_{\mathrm{II}} 3^{-1}$\\
139 & $C_{6}$ & 0 & 6 & 7 & $0$ & $3^{-1}$ & $2^{-6}_{\mathrm{II}}$ & $2^{-6}_{\mathrm{II}} 3^{-1}$\\
140 & $C_{3}$ & 0 & 3 & 6 & $0$ & $3^{-1}$ & $3^{+7}$ & $3^{+8}$\\
141 & $C_{6}$ & 0 & 6 & 5 & $0$ & $3^{-1}$ & $2^{+4}_{\mathrm{II}} 3^{+4}$ & $2^{+4}_{\mathrm{II}} 3^{-5}$\\
142 & $C_{6}$ & 0 & 6 & 4 & $0$ & $3^{-1}$ & $2^{-6}_{\mathrm{II}} 3^{-5}$ & $2^{-6}_{\mathrm{II}} 3^{-6}$\\
143 & $C_{6}$ & 0 & 6 & 3 & $0$ & $3^{-1}$ & $2^{-6}_{\mathrm{II}} 3^{+4}$ & $2^{-6}_{\mathrm{II}} 3^{-5}$\\
144 & $C_{8}$ & 0 & 8 & 3 & $0$ & $3^{-1}$ & $2^{-6}_{\mathrm{II}}$ & $2^{-6}_{\mathrm{II}} 3^{-1}$\\
145 & $C_{8}$ & 0 & 8 & 3 & $0$ & $3^{-1}$ & $2^{-6}_{\mathrm{II}}$ & $2^{-6}_{\mathrm{II}} 3^{-1}$\\
146 & $C_{8}$ & 0 & 8 & 3 & $0$ & $3^{-1}$ & $2^{-6}_{\mathrm{II}}$ & $2^{-6}_{\mathrm{II}} 3^{-1}$\\
147 & $C_{12}$ & 0 & 12 & 3 & $0$ & $3^{-1}$ & $2^{-6}_{\mathrm{II}}$ & $2^{-6}_{\mathrm{II}} 3^{-1}$\\
148 & $C_{12}$ & 0 & 12 & 3 & $0$ & $3^{-1}$ & $2^{-6}_{\mathrm{II}}$ & $2^{-6}_{\mathrm{II}} 3^{-1}$\\
149 & $C_{12}$ & 0 & 12 & 2 & $0$ & $3^{-1}$ & $2^{+4}_{\mathrm{II}} 3^{+4}$ & $2^{+4}_{\mathrm{II}} 3^{-5}$\\
150 & $C_{12}$ & 0 & 12 & 2 & $0$ & $3^{-1}$ & $2^{+4}_{\mathrm{II}} 3^{+4}$ & $2^{+4}_{\mathrm{II}} 3^{-5}$\\
151 & $C_{16}$ & 0 & 16 & 1 & $0$ & $3^{-1}$ & $2^{-6}_{\mathrm{II}}$ & $2^{-6}_{\mathrm{II}} 3^{-1}$\\
152 & $C_{16}$ & 0 & 16 & 1 & $0$ & $3^{-1}$ & $2^{-6}_{\mathrm{II}}$ & $2^{-6}_{\mathrm{II}} 3^{-1}$\\
153 & $C_{24}$ & 0 & 24 & 1 & $0$ & $3^{-1}$ & $2^{-6}_{\mathrm{II}}$ & $2^{-6}_{\mathrm{II}} 3^{-1}$\\
154 & $C_{24}$ & 0 & 24 & 1 & $0$ & $3^{-1}$ & $2^{-6}_{\mathrm{II}}$ & $2^{-6}_{\mathrm{II}} 3^{-1}$\\
155 & $C_{32}$ & 0 & 32 & 0 & $0$ & $3^{-1}$ & $2^{-6}_{\mathrm{II}}$ & $2^{-6}_{\mathrm{II}} 3^{-1}$\\
156 & $C_{48}$ & 0 & 48 & 0 & $0$ & $3^{-1}$ & $2^{-6}_{\mathrm{II}}$ & $2^{-6}_{\mathrm{II}} 3^{-1}$\\
\hline
\end{longtable}
\endgroup

\bibliography{reference}

@article {Matsumura1963OnTA,
    AUTHOR = {Matsumura, H. and Monsky, P.},
     TITLE = {On the automorphisms of hypersurfaces},
   JOURNAL = {J. Math. Kyoto Univ.},
    VOLUME = {3},
      YEAR = {1963/64},
     PAGES = {347--361},
       DOI = {10.1215/kjm/1250524785},
}

@article {voisin1986torelli,
    AUTHOR = {Voisin, C.},
     TITLE = {Th\'eor\`eme de {T}orelli pour les cubiques de {$\mathbb P^5$}},
   JOURNAL = {Invent. Math.},
    VOLUME = {86},
      YEAR = {1986},
    NUMBER = {3},
     PAGES = {577--601},
       DOI = {10.1007/BF01389270},
}

@article {voisin2008erratum,
    AUTHOR = {Voisin, C.},
     TITLE = {Erratum: ``{A} {T}orelli theorem for cubics in {$\mathbb P^5$}''},
   JOURNAL = {Invent. Math.},
    VOLUME = {172},
      YEAR = {2008},
    NUMBER = {2},
     PAGES = {455--458},
       DOI = {10.1007/s00222-008-0116-z},
}

@article {Nikulin_1980,
    AUTHOR = {Nikulin, V. V.},
     TITLE = {Integral symmetric bilinear forms and some of their applications},
   JOURNAL = {Math. USSR-Izv.},
    VOLUME = {14},
      YEAR = {1980},
    NUMBER = {1},
     PAGES = {103--167},
       DOI = {10.1070/IM1980v014n01ABEH001060},
}

@article {zheng2021orbifold,
    AUTHOR = {Zheng, Zhiwei},
     TITLE = {Orbifold aspects of certain occult period maps},
   JOURNAL = {Nagoya Math. J.},
    VOLUME = {243},
      YEAR = {2021},
     PAGES = {137--156},
       DOI = {10.1017/nmj.2019.36},
}

@article {griffiths1969rational,
    AUTHOR = {Griffiths, Phillip A.},
     TITLE = {On the periods of certain rational integrals. {I}},
   JOURNAL = {Ann. of Math. (2)},
    VOLUME = {90},
      YEAR = {1969},
    NUMBER = {3},
     PAGES = {460--495},
       DOI = {10.2307/1970746},
}

@article {hassett2000special,
    AUTHOR = {Hassett, B.},
     TITLE = {Special cubic fourfolds},
   JOURNAL = {Compositio Math.},
    VOLUME = {120},
      YEAR = {2000},
    NUMBER = {1},
     PAGES = {1--23},
       DOI = {10.1023/A:1001706324425},
}

@article {laza2010period,
    AUTHOR = {Laza, R.},
     TITLE = {The moduli space of cubic fourfolds via the period map},
   JOURNAL = {Ann. of Math. (2)},
    VOLUME = {172},
      YEAR = {2010},
    NUMBER = {1},
     PAGES = {673--711},
       DOI = {10.4007/annals.2010.172.673},
}

@article {looi2009period,
    AUTHOR = {Looijenga, E.},
     TITLE = {The period map for cubic fourfolds},
   JOURNAL = {Invent. Math.},
    VOLUME = {177},
      YEAR = {2009},
    NUMBER = {1},
     PAGES = {213--233},
       DOI = {10.1007/s00222-009-0178-6},
}

@article {yu2020moduli,
    AUTHOR = {Yu, C. and Zheng, Z.},
     TITLE = {Moduli spaces of symmetric cubic fourfolds and locally symmetric varieties},
   JOURNAL = {Algebra Number Theory},
    VOLUME = {14},
      YEAR = {2020},
    NUMBER = {10},
     PAGES = {2647--2683},
       DOI = {10.2140/ant.2020.14.2647},
}

@article {laza2022automorphisms,
    AUTHOR = {Laza, Radu and Zheng, Zhiwei},
     TITLE = {Automorphisms and periods of cubic fourfolds},
   JOURNAL = {Math. Z.},
  FJOURNAL = {Mathematische Zeitschrift},
    VOLUME = {300},
      YEAR = {2022},
    NUMBER = {2},
     PAGES = {1455--1507},
      ISSN = {0025-5874,1432-1823},
       DOI = {10.1007/s00209-021-02810-x},
       URL = {https://doi.org/10.1007/s00209-021-02810-x},
}

@article{hohn2016290,
  author={H{\"o}hn, Gerald and Mason, Geoffrey},
  title={The 290 fixed-point sublattices of the {Leech} lattice},
  journal={Journal of Algebra},
  volume={448},
  pages={618--637},
  year={2016},
  doi={10.1016/j.jalgebra.2015.08.028},
  url={https://doi.org/10.1016/j.jalgebra.2015.08.028}
}

@article {KOIKE202512,
    AUTHOR = {Koike, Kenji},
     TITLE = {Cubic fourfolds with symplectic automorphisms},
   JOURNAL = {J. Algebra},
  FJOURNAL = {Journal of Algebra},
    VOLUME = {680},
      YEAR = {2025},
     PAGES = {12--57},
      ISSN = {0021-8693,1090-266X},
       DOI = {10.1016/j.jalgebra.2025.04.037},
       URL = {https://doi.org/10.1016/j.jalgebra.2025.04.037},
}

@article {KOIKE2026,
    AUTHOR = {Koike, Kenji},
     TITLE = {Corrigendum to ``{C}ubic fourfolds with symplectic automorphisms''},
   JOURNAL = {J. Algebra},
  FJOURNAL = {Journal of Algebra},
    VOLUME = {705},
      YEAR = {2026},
     PAGES = {358--362},
      ISSN = {0021-8693,1090-266X},
       DOI = {10.1016/j.jalgebra.2026.05.021},
       URL = {https://doi.org/10.1016/j.jalgebra.2026.05.021},
}

@misc {FWZ26,
    AUTHOR = {Fu, Jie and Wang, Shihao and Zheng, Zhiwei},
     TITLE = {Non-symplectic indices of automorphism groups of smooth cubic fourfolds},
      YEAR = {2026},
      NOTE = {arXiv:2606.11754},
       URL = {https://arxiv.org/abs/2606.11754},
}

@misc {FZ26,
    AUTHOR = {Fu, Jie and Zheng, Zhiwei},
     TITLE = {Automorphism groups of smooth cubic fourfolds through lattice theory},
      YEAR = {2026},
      NOTE = {arXiv:2609.06683},
       URL = {https://arxiv.org/abs/2609.06683},
}

@misc {FWZclass,
    AUTHOR = {Fu, Jie and Wang, Shihao and Zheng, Zhiwei},
     TITLE = {Classification of automorphism groups of smooth cubic threefolds and fourfolds},
      YEAR = {2026},
      NOTE = {arXiv:2609.20403},
       URL = {https://arxiv.org/abs/2609.20403},
}

@misc {MarquandInvolution,
    AUTHOR = {Marquand, Lisa},
     TITLE = {Cubic fourfolds with an involution},
      YEAR = {2022},
      NOTE = {arXiv:2202.13213},
       URL = {https://arxiv.org/abs/2202.13213},
}

@article {yang2024automorphism,
    AUTHOR = {Yang, Song and Yu, Xun and Zhu, Zigang},
     TITLE = {Automorphism groups of cubic fivefolds and fourfolds},
   JOURNAL = {J. Lond. Math. Soc. (2)},
    VOLUME = {110},
      YEAR = {2024},
    NUMBER = {4},
     PAGES = {Paper No. e12997, 35},
       DOI = {10.1112/jlms.12997},
       URL = {https://doi.org/10.1112/jlms.12997},
}

@misc {he2025cubicfourfoldsorder7automorphism,
    AUTHOR = {He, Xuancong and Li, Yi and Wang, Shihao and Zheng, Zhiwei},
     TITLE = {Cubic fourfolds with an order-$7$ automorphism},
      YEAR = {2025},
      NOTE = {arXiv:2509.26359v1},
       URL = {https://arxiv.org/abs/2509.26359},
}

@manual {OSCAR,
       KEY = {OSCAR},
ORGANIZATION = {The OSCAR Team},
     TITLE = {{OSCAR} -- Open Source Computer Algebra Research System, Version 1.8.2},
      YEAR = {2026},
       URL = {https://www.oscar-system.org},
}

@misc{FuWangZhengAutcub4fold2026,
  author       = {Fu, Jie and Wang, Shihao and Zheng, Zhiwei},
  title        = {Computational materials for automorphism groups of smooth cubic threefolds and fourfolds},
  year         = {2026},
  howpublished = {Zenodo},
  note         = {Version v2.1.0, \url{https://doi.org/10.5281/zenodo.23437698}},
  doi          = {10.5281/zenodo.23437698},
  url          = {https://zenodo.org/records/23437698}
}

\end{document}